\documentclass[a4paper,11pt]{article}
\usepackage{amsmath,indentfirst,amsfonts,dsfont,amssymb,amsthm,graphicx,enumerate,mathrsfs,color,bm}
\usepackage{algorithm}               
\usepackage{algcompatible}             
\usepackage{url}
\numberwithin{equation}{section}
\numberwithin{algorithm}{section}

\newtheorem{definition}{Definition}[section]
\newtheorem{theorem}{Theorem}[section]
\newtheorem{lemma}{Lemma}[section]

\newtheorem{corollary}{Corollary}[section]
\newtheorem{problem}{Problem}[section]

\newtheorem{remark}{Remark}[section]

\title{\bf Randomized Algorithms for the Low Multilinear Rank Approximations of Tensors}
\author{Maolin Che\thanks{E-mail: chncml@outlook.com and cheml@swufe.edu.cn. School of
Economic Mathematics, Southwestern University of Finance and Economics, Chengdu, 611130, P. R. of
China. This author is supported by the National Natural Science
Foundation of China under grant 11901471.}
\and Yimin Wei\thanks{E-mail: ymwei@fudan.edu.cn and yimin.wei@gmail.com. School of
Mathematical Sciences and Key Laboratory of Mathematics for
Nonlinear   Sciences,  Fudan University, Shanghai, 200433, P. R. of
China. This author is supported by the National Natural Science
Foundation of China under grant 11771099 and Innovation Program of Shanghai Municipal Education Commission.}
\and Hong Yan\thanks{E-mail: h.yan@cityu.edu.hk. Department of Electrical Engineering, City University of Hong Kong, 83 Tat Chee Avenue, Kowloon, Hong Kong. This author is supported by the Hong Kong Research Grants Council (Project C1007-15G) and and City University of Hong Kong (Project 9610308).}}

\begin{document}

\maketitle

\begin{abstract}
In this paper, we focus on developing  randomized algorithms for the computation of low multilinear rank approximations of tensors based on the random projection and the singular value decomposition.  Following the theory of the singular values of sub-Gaussian matrices, we make a  probabilistic analysis for the error bounds for the randomized algorithm. We demonstrate the effectiveness of proposed algorithms via several numerical examples.
  \bigskip

  {\bf Keywords:} Randomized algorithms; low multilinear rank approximation; sub-Gaussian matrices; singular value decomposition; singular values.
  \bigskip

  {\bf AMS subject classifications:} 15A18, 15A69, 65F15, 65F10
\end{abstract}

\section{Introduction}
\label{SUB:sect1}

A wide range of  applications, such as in chemometrics, signal processing and high order statistics \cite{cichocki2015tensor,nonnegative,tensor_app,Kolda,vervliet2014breaking}, involve the manipulation of quantities with elements addressed by more than two indices. With three indices or more, these higher-order expansions of vectors (first-order) and matrices (second-order) are called higher-order tensors, multidimensional matrices, or multiway arrays.

We use the symbol $\mathcal{A}\in\mathbb{R}^{I_1\times I_2\times I_3}$ to represent a three-dimensional array of real numbers with entries given by
$a_{i_1i_2 i_3}\in\mathbb{R}$ for $i_n=1,2,\dots,I_n$ and $n=1,2,3$. For notational simplicity, we illustrate our results by using third-order tensors whenever generalizations to higher-order cases are straightforward. Slight differences will be explained when needed.

In this paper, we consider the low multilinear rank approximation of a tensor, which is defined as follows.
\begin{problem}\label{SUB:prob1}
	Suppose that $\mathcal{A}\in\mathbb{R}^{I_1\times I_2\times I_3}$. The goal is to require three orthonormal matrices ${\bf Q}_n\in\mathbb{R}^{I_{n}\times \mu_{n}}$ with $\mu_n\leq I_n$, such that
	\begin{equation*}
	a_{i_1i_2i_3}\approx\sum_{j_1,j_2,j_3=1}^{I_1,I_2,I_3}a_{j_1j_2j_3}p_{1,i_1j_1}p_{2,i_2j_2}p_{3,i_3j_3},
	\end{equation*}
	where $\mathbf{P}_n=\mathbf{Q}_n\mathbf{Q}_n^\top\in\mathbb{R}^{I_n\times I_n}$ is a projection matrix and $p_{n,i_nj_n}$ is the $(i_n,j_n)$-element of $\mathbf{P}_n$.
\end{problem}

Problem \ref{SUB:prob1} can be solved by a number of recently developed algorithms, such as higher-order orthogonal iteration \cite{lathauwer_rank_one_2000}, the Newton-Grassmann method \cite{elden_newton_2009}, the Riemannian
trust-region method \cite{trust_region_2011}, the Quasi-Newton method \cite{quasi_app_2010}, semi-definite programming (SDP) \cite{semidefinite_2009}, and Lanczos-type iteration \cite{GOS_SISC_2012,SE_T_2009}. The readers can refer to two  surveys \cite{tensor_survery_Grasedyck,Kolda} for the relevant information. If the columns of each $\mathbf{Q}_n$ are extracted from the mode-$n$ unfolding matrix $\mathbf{A}_{(n)}$, then the solution of Problem \ref{SUB:prob1} is called as the {\it CUR}-type decomposition of $\mathcal{A}$, which can be obtained by the different versions of the cross approximation method.  We  refer to \cite{tensor_cur_2010,tensor_based_2007_laa,GT_SMMCS_2001,tensor_cur_2008,OST_SIMAX_2008,OST_NLAA_2010} for more details about a {\it CUR}-type decomposition of tensors. On the other hand, for Problem \ref{SUB:prob1}, when we restrict the entries of the tensor $\mathcal{A}$ and the matrices $\mathbf{Q}_n$ to be nonnegative and admit the matrices $\mathbf{Q}_n$ not being orthonormal, the solution of Problem \ref{SUB:prob1} is  called a nonnegative Tucker decomposition  \cite{FH_CMS_2008,nonegative_tucker_zhang2016,nonnegative_tucker_zhou2012,nonnegative_tucker_zhou2015}.

Low-rank matrix approximations, such as the truncated singular value decomposition \cite[page 291]{Golub} and the rank-revealing QR decomposition \cite{rank_QR_1987}, play a central role in data analysis and scientific computing. Halko {\it et al.}  \cite{random_2011_siamreview} present a modular framework to construct randomized algorithms for computing partial matrix decompositions. Randomized algorithms for low-rank approximations and their theory have been well established in terms of its near optimality compared to the Eckart-Young theorem. We  recommend three surveys \cite{drineas2016randnla,mahoney2011randomized,woodruff2014sketching} for more details about the randomized algorithms for computing low rank matrix approximations.

Randomized algorithms have recently been applied to tensor decompositions. Drineas and Mahoney \cite{tensor_based_2007_laa} present and analyze randomized algorithms for computing the CUR-type decomposition of a tensor, which can be viewed as the generalization of the Linear-Time-SVD algorithm \cite{Drineas2005fast} and the Fast-Approximate-SVD algorithm \cite{deshpande2006matrix} for the low-rank approximations of matrices to tensors, which are originally for matrices. Battaglino {\it et al.} \cite{Battaglino2017a} extend randomized least squares methods to tensors and show the workload of CANDECOMP/PARAFAC-ALS can be drastically reduced without sacrifice in quality. Vervliet and De Lathauwer \cite{vervliet2016a} present the randomized block sampling canonical polyadic decomposition method, which combines increasingly popular ideas from randomization and stochastic optimization to tackle the computational problems.

Zhou {\it et al.} \cite{decomposition_big_tensor} propose a distributed randomized Tucker decomposition for arbitrarily big
tensors but with relatively low multilinear rank. Che and Wei \cite{che2018randomized} design adaptive randomized algorithms for computing the low multilinear rank approximation of tensors and the approximate tensor train decomposition. More results about this topic can be found in \cite{biagioni2015randomized,tensor_sparsification_2015,tensor_fast_2010} and their references. More recently, many researchers propose randomized algorithms for low multilinear rank approximations \cite{ahmadiasl2020randomized,che2020theory,kressner2017recompression,
minster2019randomized,sun2019low,WCW2020}.

Suppose that $\mu_n$ is a given positive integer and $R$ is the oversampling parameter. In the work of \cite{decomposition_big_tensor}, the column space of each mode-$n$ unfolding of $\mathbf{A}_{(n)}$ is approximated by that of $\mathbf{A}_{(n)}\bm{\Omega}_n$, where $\bm{\Omega}_n\in\mathbb{R}^{I^2\times (\mu_n+R)}$ is a standard Gaussian matrix and $\mathbf{A}_{(n)}$ is the mode-$n$ unfolding of $\mathcal{A}\in\mathbb{R}^{I\times I\times I}$. However, the column space of each mode-$n$ unfolding of $\mathbf{A}_{(n)}$ is approximated by that of $\mathbf{A}_{(n)}\bm{\Omega}_n$ with $\bm{\Omega}_n=(\bm{\Omega}_{n,1}\odot\bm{\Omega}_{n,2})$ in \cite{che2018randomized}, where $\bm{\Omega}_{n,1},\bm{\Omega}_{n,2}\in\mathbb{R}^{I\times (\mu_n+R)}$ are standard Gaussian matrices. The difference between \cite{che2018randomized} and \cite{decomposition_big_tensor} is that the storage capacity of $\bm{\Omega}_n$ is different. As shown in \cite{che2018randomized,decomposition_big_tensor}, comparison with the deterministic algorithms for low multilinear rank approximations, randomized algorithms are often faster and more robust and the algorithm in \cite{che2018randomized} is faster than that of \cite{decomposition_big_tensor}.

 The main contribution of this paper is to design a more effective randomized algorithm for the computation of low multilinear rank approximations of tensors. Our proposed algorithm can be divided into two stages. Suppose that $\mathcal{A}\in\mathbb{R}^{I_1\times I_2\times I_3}$. In the first stage, for each $n$, the Kronecker product of two standard Gaussian matrices of suitable dimensions are applied to the mode-$n$ unfolding of $\mathcal{A}$, which is an $I_n\times \prod_{m=1,m\neq n}^3L_{n,m}$ matrix $\mathbf{B}_{n,(n)}$. In the second stage, we utilize the  singular value decomposition (SVD) to obtain an orthonormal matrix, satisfying the requirement that the column space of the matrix can be used to approximate $\mathbf{B}_{n,(n)}$. Note that Algorithm \ref{SUB:alg3} can be viewed as the generalization of the core idea of the randomized algorithm in \cite{martinsson2011a}. As shown in Section \ref{SUB:sect6}, in terms of CPU times, the proposed algorithm is faster than the existing algorithms for low multilinear rank approximations; and in terms of RLNE, the proposed algorithms are often more accurate than the existing algorithms.


Throughout this paper, we assume that $I$, $J$, and $N$ denote the index upper bounds, unless stated otherwise. We adopt lower case letters $x,u,v,\dots$ for scalars, lower case bold letters $\mathbf{x},\mathbf{u},
\mathbf{v},\dots$ for vectors, bold capital letters $\mathbf{A},\mathbf{B},\mathbf{C},\dots$ for matrices, and calligraphic letters $\mathcal{A},\mathcal{B},\mathcal{C},\dots$ for tensors. This notation is consistently used for lower-order parts of a given structure. For example, the entry with row index $i$ and column index $j$ in a matrix ${\bf A}$, i.e., $({\bf A})_{ij}$, is represented as $a_{ij}$ (also $(\mathbf{x})_i=x_i$ and $(\mathcal{A})_{i_1i_2i_3}=a_{i_1i_2i_3}$).
For a vector $\mathbf{x}\in\mathbb{R}^{I}$, we use $\|\mathbf{x}\|_2$ and $\mathbf{x}^\top$ to denote its 2-norm and transpose, respectively. $\mathbf{0}$ denotes the zero vector in $\mathbb{R}^{I}$.  $\mathbf{A} \otimes \mathbf{B}$
denotes the Kronecker product of matrices
$\mathbf{A}\in\mathbb{R}^{I\times J}$ and $\mathbf{B}\in\mathbb{R}^{K\times L}$.  $\mathbf{A} \odot \mathbf{B}$ is the Khatri-Rao product of matrices
$\mathbf{A}\in\mathbb{R}^{I\times L}$ and $\mathbf{B}\in\mathbb{R}^{J\times L}$. $\mathbf{A}^{\dag}$ represents the Moore-Penrose pseudoinverse of $\mathbf{A}\in\mathbb{R}^{I\times J}$. A matrix $\mathbf{Q}\in\mathbb{R}^{I\times K}$ with $K<I$ is orthonormal if $\mathbf{Q}^\top\mathbf{Q}=\mathbf{I}_K$.

The rest of our paper is organized as follows. In Section \ref{SUB:sect2}, we introduce basic tensor operations and singular values of random matrices. We present the higher-order singular value decomposition and higher-order orthogonal iteration for the low multilinear rank approximation in Section \ref{SUB:sect3}. The randomized algorithms for the low multilinear rank approximation are presented in Section \ref{SUB:sect4}. In the same section, we provide  probabilistic error bounds and analyze computational complexity of these three algorithms. The  error bounds are analyzed in Section \ref{SUB:sect5}. We illustrate our algorithms via numerical examples in Section \ref{SUB:sect6}. We conclude this paper and discuss future research topics in Section \ref{SUB:sect7}.
\section{Preliminaries}
\label{SUB:sect2}

We introduce the basic notations and concepts involving tensors which will be used in this paper.
The mode-$n$ product \cite{nonnegative,Kolda} of a real tensor $\mathcal{A}\in \mathbb{R}^{I_1\times I_2\times I_3}$ by a matrix ${\bf B}\in \mathbb{R}^{J\times I_n}$, denoted
by $\mathcal{C}=\mathcal{A}\times_{n}{\bf B}$:
\begin{equation*}
\begin{split}
n=1:\ c_{ji_2i_3}=\sum_{i_1=1}^{I_1}a_{i_1i_2i_3}b_{ji_1};\
n=2:\ c_{i_1ji_3}=\sum_{i_2=1}^{I_2}a_{i_1i_2i_3}b_{ji_2};\
n=3:\ c_{i_1i_2j}=\sum_{i_3=1}^{I_3}a_{i_1i_2i_3}b_{ji_3}.
\end{split}
\end{equation*}


For a tensor $\mathcal{A}\in \mathbb{R}^{I_1\times I_2\times I_3}$ and three matrices $\mathbf{F}\in \mathbb{R}^{J_n\times I_n}$, $\mathbf{G}\in \mathbb{R}^{J_m\times I_m}$ and $\mathbf{H}\in\mathbb{R}^{J'_n\times J_n}$, one has \cite{Kolda}
\begin{equation*}
\begin{cases}
&(\mathcal{A}\times_{n}\mathbf{F})\times_{m}\mathbf{G}
=(\mathcal{A}\times_{m}\mathbf{G})\times_{n}\mathbf{F}=\mathcal{A}\times_{n}\mathbf{F}\times_{m}\mathbf{G},\\ &(\mathcal{A}\times_{n}\mathbf{F})\times_{n}\mathbf{H}=\mathcal{A}\times_{n}(\mathbf{H}\cdot \mathbf{F}),
\end{cases}
\end{equation*}
where `$\cdot$' represents the multiplication of two matrices.

 For two tensors $\mathcal{A},\mathcal{B}\in \mathbb{R}^{I_1\times I_2\times I_3}$, the {\it Frobenius norm} of a tensor $\mathcal{A}$ is given by  $\|\mathcal{A}\|_{F}=\sqrt{\langle\mathcal{A},\mathcal{A}\rangle}$ and the scalar product $\langle\mathcal{A},\mathcal{B}\rangle$ is defined as \cite{hosvd,Kolda}
$$\langle\mathcal{A},\mathcal{B}\rangle=\sum_{i_1,i_2,i_3=1}^{I_1,I_2,I_3}a_{i_{1}i_{2} i_{3}}b_{i_{1}i_{2}i_{3}}.$$

  The mode-$n$ unfolding matrix of a third-order tensor can be understood as the process of the construction of a matrix containing all the mode-$n$ vectors of the tensor. The order of the columns is not unique and the unfolding matrix of $\mathcal{A}\in \mathbb{R}^{I_1\times I_2\times I_3}$, denoted by ${\bf A}_{(n)}$, arranges the mode-$n$ fibers into columns of this matrix. More specifically, a tensor element $(i_1,i_2,i_3)$ maps on a matrix element $(i_n,j)$, where
\begin{equation*}
\begin{split}
n=1:\ j=i_2+(i_3-1)I_2;\quad n=2:\ j=i_1+(i_3-1)I_1;\quad n=3:\ j=i_1+(i_2-1)I_1.
\end{split}
\end{equation*}

\subsection{Singular values of random matrices}

We first review the definition of the sub-Gaussian random variable. Sub-Gaussian variables are an important class of random variables that have strong tail decay properties.
\begin{definition}{{\bf (\cite[Definition 3.2]{shabat2016randomized})}}
	\label{SUB:def4}
	A real valued random variable $X$ is called a sub-Gaussian random variable if there exist $b>0$ such that for all $t>0$ we have $\mathbf{E}(e^{tX})\leq e^{b^2t^2/2}$. A random variable $X$ is centered if $\mathbf{E}(X)=0$.
\end{definition}

We cite several results adapted from \cite{litvak2005smallest,rudelson2009smallest} about random matrices whose entries are sub-Gaussian. We emphasize the case where $\mathbf{A}$ is an $I\times J$ matrix with $J>(1+1/\ln(I))I$. Similar results can be found in \cite{litvak2012smallest} for the square and almost square matrices.
\begin{definition}
	\label{SUB:def5}
	Assume that $\mu\geq1$, $a_1>0$ and $a_2>0$. The set $\mathbb{A}(\mu,a_1,a_2,I,J)$ consists of all $I\times J$ random matrices $\mathbf{A}$ whose entries are the centered  independent identically distributed real valued random variables satisfying the following conditions: {\rm (a)} moments: $\mathbf{E}(|a_{ij}|^3)\leq \mu^3$; {\rm (b)} norm: $\mathbf{P}(\|\mathbf{A}\|_2>a_1\sqrt{J})\leq e^{-a_2J}$; {\rm (c)} variance: $\mathbf{E}(|a_{ij}|^2)\leq 1$.
\end{definition}

It is proven in \cite{litvak2005smallest} that if $\mathbf{A}$ is sub-Gaussian, then $\mathbf{A}\in\mathbb{A}(\mu,a_1,a_2,I,J)$. For a Gaussian matrix with zero mean and unit variance, we have $\mu=(4/\sqrt{2\pi})^{1/3}$.
\begin{theorem}{{\bf (\cite[Section 2]{litvak2005smallest})}}
	\label{SUB:thm4}
	Suppose that $\mathbf{A}\in \mathbb{R}^{I\times J}$ is sub-Gaussian with $I\leq J$, $\mu\geq1$ and $a_2>0$. Then
$
	\mathbf{P}(\|\mathbf{A}\|_2>a_1\sqrt{J})\leq e^{-a_2J}
$,
	where $a_1=6\mu\sqrt{a_2+4}$.
\end{theorem}

Theorem \ref{SUB:thm4} establishes an upper bound for the largest singular value that depends on the desired probability. Theorem \ref{SUB:thm5}  bounds from the upper below the smallest singular value of a random sub-Gaussian matrices.
\begin{theorem}{{\bf (\cite[Section 2]{litvak2005smallest})}}
	\label{SUB:thm5}
	Let $\mu\geq1$, $a_1>0$ and $a_2>0$. Suppose that $\mathbf{A}\in\mathbb{A}(\mu,a_1,a_2,I,J)$ with $J>(1+1/\ln(I))I$. Then, there exist positive constants $c_1$ and $c_2$ such that
	\begin{equation*}
	\mathbf{P}(\sigma_I(\mathbf{A})\leq c_1\sqrt{J})\leq e^{-J}+e^{-c''J/(2\mu^6)}+e^{-a_2J}\leq e^{-c_2J}.
	\end{equation*}
\end{theorem}
\begin{remark}
	\label{SUB:rem3}
	For Theorem {\rm\ref{SUB:thm5}}, the exact values of constants $c_1$, $c_2$ and $c''$ are discussed in {\rm \cite{shabat2016randomized}}.
\end{remark}
\section{HOSVD and HOOI}
\label{SUB:sect3}

A {\it Tucker decomposition} \cite{tucker_1966} of a tensor $\mathcal{A}\in \mathbb{R}^{I_1\times I_2 \times I_3}$ is defined as
\begin{equation}\label{SUB:eqn4}
\mathcal{A}\approx\mathcal{G}\times_1{\bf U}^{(1)}\times_2{\bf U}^{(2)}\times_3{\bf U}^{(3)},
\end{equation}
where ${\bf U}^{(n)}\in \mathbb{R}^{I_n\times R_n}$ are called the {\it mode-$n$ factor matrices} and $\mathcal{G}\in \mathbb{R}^{R_1\times R_2\times R_3}$ is called the {\it core tensor} of the decomposition with the set $\{R_1,R_2,R_3\}$.

The Tucker decomposition is closely related to the mode-$n$ unfolding matrix $\mathbf{A}_{(n)}$ with $n=1,2,3$. In particular, the relation (\ref{SUB:eqn4}) implies
$$
\begin{cases}
{\bf A}_{(1)}&\approx{\bf U}^{(1)}{\bf G}_{(1)}({\bf U}^{(2)}\otimes {\bf U}^{(3)})^{\top};\\
{\bf A}_{(2)}&\approx{\bf U}^{(2)}{\bf G}_{(2)}({\bf U}^{(1)}\otimes {\bf U}^{(3)})^{\top};\\
{\bf A}_{(3)}&\approx{\bf U}^{(3)}{\bf G}_{(3)}({\bf U}^{(1)}\otimes {\bf U}^{(2)})^{\top}.
\end{cases}
$$
It follows that the rank of ${\bf A}_{(n)}$ is less than or equal to $R_n$, as the mode-$n$ factor ${\bf U}^{(n)}\in\mathbb{R}^{I_n\times R_n}$  at most has rank  $R_n$. We define the multilinear rank of $\mathcal{A}$ as the tuple
$\{R_1,R_2,R_3\}$, where the rank of ${\bf A}_{(n)}$ is equal to $R_n$.
%

 Applying the singular value decomposition (SVD) to $\mathbf{A}_{(n)}$ with $n=1,2,3$, we obtain a special form of the Tucker decomposition of a given tensor, which is called the {\it higher-order singular value decomposition} (HOSVD) \cite{hosvd}. 

When $R_n<{\rm rank}(\mathbf{A}_{(n)})$ for one or more $n$, the decomposition is called the {\it truncated HOSVD}. The truncated HOSVD is not optimal in terms of giving the best fitting as measured by the Frobenius norm of the difference, but it is used to initialize iterative algorithms to compute the best approximation of a specified multilinear rank \cite{lathauwer_rank_one_2000,elden_newton_2009,trust_region_2011,quasi_app_2010}.  For given three positive integers $\mu_1$, $\mu_2$ and $\mu_3$, the low multilinear rank approximation of $\mathcal{A}$ can be rewritten as the optimization problem
respect  to the Frobenius norm
\begin{equation*}
\begin{split}
\min_{\mathcal{G},\mathbf{Q}_1,\mathbf{Q}_2,\mathbf{Q}_3}&\quad
\left\|\mathcal{A}-\mathcal{G}\times_1\mathbf{Q}_1\times_2\mathbf{Q}_2
\times_3\mathbf{Q}_3\right\|_F^2,\\
\text{subject to}&\quad\mathcal{G}\in\mathbb{R}^{\mu_1\times \mu_2\times \mu_3},
\quad\mathbf{Q}_n\in\mathbb{R}^{I_n\times \mu_n}\text{ is orthonormal}.
\end{split}
\end{equation*}

If $\mathbf{Q}_n^*$ is a solution of the above maximization problem, then we call $\mathcal{A}\times_1\mathbf{P}_1\times_2\mathbf{P}_2
\times_3\mathbf{P}_3$ as a {\it low multilinear rank approximation} of $\mathcal{A}$, where $\mathbf{P}_n=\mathbf{Q}_n^*(\mathbf{Q}_n^*)^\top$.

\section{The proposed algorithm and its analysis}\label{SUB:sect4}

In this section, we present our randomized algorithm for the low multilinear rank approximations of tensors, summarized in Algorithm \ref{SUB:alg3}. We give a slight modification of Algorithm \ref{SUB:alg3} to reduce its computational complexity.

\subsection{Framework for the algorithm}

For each $n$, Algorithm \ref{SUB:alg3} begins by projecting the mode-$n$ unfolding of the input tensor on the Kronecker product of random matrices. The result matrix captures most of the range of the mode-$n$ unfolding of the tensor. Then we compute a basis for this matrix by Lemma \ref{SUB:lem10}. Finally, we project the input tensor on it.
\begin{algorithm}[htb]
	\caption{The proposed randomized algorithm for low multilinear rank approximations with $N=3$}
	\label{SUB:alg3}
	\begin{algorithmic}[1]
		\STATEx {\bf Input}: A tensor $\mathcal{A}\in \mathbb{R}^{I_1\times I_2\times I_3}$ to decompose, the desired multilinear rank $\{\mu_1,\mu_2,\mu_3\}$, $L_{3,1}L_{3,2}\geq\mu_3+K$, $L_{2,1}L_{2,3}\geq\mu_2+K$, and $L_{1,2}L_{1,3}\geq\mu_1+K$, where $K$ is a oversampling parameter.
		\STATEx {\bf Output}: Three orthonormal matrices $\mathbf{Q}_n$ such that $\|\mathcal{A}\times_1 (\mathbf{Q}_1\mathbf{Q}_1^\top)\times_2 (\mathbf{Q}_2\mathbf{Q}_2^\top)
		\times_3 (\mathbf{Q}_3\mathbf{Q}_3^\top)-\mathcal{A}\|_F\leq
		\sum_{n=1}^3O(\Delta_{\mu_n+1}(\mathbf{A}_{(n)}))$.
		\STATE Form six real matrices $\mathbf{G}_{n,m}\in\mathbb{R}^{L_{n,m}\times I_m}$ whose entries are independent and identically distributed (i.i.d.) Gaussian random variables of zero mean and unit variance, where $m,n=1,2,3$ and $m\neq n$.
		\STATE Compute three product tensors
		\begin{equation*}
		\mathcal{B}_1=\mathcal{A}\times_2\mathbf{G}_{1,2}\times_3\mathbf{G}_{1,3},\quad
		\mathcal{B}_2=\mathcal{A}\times_1\mathbf{G}_{2,1}\times_3\mathbf{G}_{2,3},\quad
		\mathcal{B}_3=\mathcal{A}\times_1\mathbf{G}_{3,1}\times_2\mathbf{G}_{3,2}.
		\end{equation*}
		\STATE Form the mode-$n$ unfolding $\mathbf{B}_{n,(n)}$ of each tensor $\mathcal{B}_n$.
		\STATE For each $\mathbf{B}_{n,(n)}$, find a real $I_n\times \mu_n$ matrix $\mathbf{Q}$ whose columns are orthonormal, such that there exists a real $\mu_n\times \prod_{m=1,m\neq n}^3L_{n,m}$ matrix $\mathbf{S}_n$ for which
		$$\|\mathbf{Q}\mathbf{S}_n-\mathbf{B}_{n,(n)}\|_2\leq\sigma_{\mu_n+1}(\mathbf{B}_{n,(n)}),$$
		where $\sigma_{\mu_n+1}(\mathbf{B}_{n,(n)})$ is the $(\mu_n+1)$st greatest singular value of $\mathbf{B}_{n,(n)}$.
		\STATE Set $\mathbf{Q}_n:=\mathbf{Q}(:,1:\mu_n)$ for all $n=1,2,3$.
	\end{algorithmic}
\end{algorithm}
\begin{remark}
	In Algorithm {\rm\ref{SUB:alg3}}, we use the computer science interpretation of $O(\cdot)$ to refer to the class of functions whose growth is bounded and below up to a constant.
\end{remark}

Suppose that three matrices $\mathbf{Q}_n\in\mathbb{R}^{I_n\times \mu_n}$ are derived from Algorithm \ref{SUB:alg3}, then we have
\begin{equation}\label{SUB:eqn5}
\begin{split}
&\mathcal{A}-\mathcal{A}\times_1({\bf Q}_1{\bf Q}_1^{\top})\times_2({\bf Q}_2{\bf Q}_2^{\top})\times_3({\bf Q}_3{\bf Q}_3^{\top})=\mathcal{A}-\mathcal{A}\times_1({\bf Q}_1{\bf Q}_1^\top)+\mathcal{A}\times_1({\bf Q}_1{\bf Q}_1^\top)\\
&-\mathcal{A}\times_1({\bf Q}_1{\bf Q}_1^\top)\times_2({\bf Q}_2{\bf Q}_2^\top)+\mathcal{A}\times_1({\bf Q}_1{\bf Q}_1^\top)\times_{2}({\bf Q}_{2}{\bf Q}_{2}^{\top})\\
&-\mathcal{A}\times_1({\bf Q}_1{\bf Q}_1^\top)\times_{2}({\bf Q}_{2}{\bf Q}_{2}^{\top})\times_3({\bf Q}_3{\bf Q}_3^{\top}).
\end{split}
\end{equation}
According to (\ref{SUB:eqn5}), we have
\begin{equation}\label{SUB:eqn13}
\left\|\mathcal{A}-\mathcal{A}\times_1({\bf Q}_1{\bf Q}_1^{\top})\times_2({\bf Q}_2{\bf Q}_2^{\top}) \times_3({\bf Q}_3{\bf Q}_3^{\top})\right\|_F^2 \leq\sum_{n=1}^3 \left\|\mathcal{A}-\mathcal{A}\times_n({\bf Q}_n{\bf Q}_n^{\top})\right\|_F^2.
\end{equation}
The result relies on the orthogonality of the projector in the Frobenius norm \cite{vvm_2012_sisc}, i.e., for any $n=1,2,3$,
\begin{equation*}
\|\mathcal{A}\|_F^2=\left\|\mathcal{A}\times_n({\bf Q}_n{\bf Q}_n^{\top})\right\|_F^2+\left\|\mathcal{A}\times_n(\mathbf{I}_{I_n}-{\bf Q}_n{\bf Q}_n^{\top})\right\|_F^2,
\end{equation*}
and the fact that  $\|{\bf AP}\|_F\leq\|{\bf A}\|_F$ with ${\bf A}\in\mathbb{R}^{I\times J}$, where the orthogonal projection $\mathbf{P}$ satisfies \cite{Golub}
\begin{equation*}
{\bf P}^2={\bf P},\quad {\bf P}^\top={\bf P},\quad {\bf P}\in\mathbb{R}^{J\times J}.
\end{equation*}
Hence, when obtaining the error bound of $\|\mathcal{A}-\mathcal{A}\times_n({\bf Q}_n{\bf Q}_n^{\top})\|_F^2$, we present an error bound for Algorithm \ref{SUB:alg3}, summarized in the following theorem.

\begin{theorem}
	\label{SUB:thm6}
	Suppose that $I_1\leq I_2I_3$, $I_2\leq I_1I_3$ and $I_3\leq I_1I_2$. Let $\mu_1$, $L_{1,2}$ and $L_{1,3}$ be integers such that $(1+1/\ln(\sqrt{\mu_1}))\sqrt{\mu_1}<L_{1,2},L_{1,3}$ and $L_{1,2}L_{1,3}<\min(I_1,I_2I_3)$.
	Let $\mu_2$, $L_{2,1}$ and $L_{2,3}$ be integers such that $(1+1/\ln(\sqrt{\mu_2}))\sqrt{\mu_2}<L_{2,1},L_{2,3}$ and $L_{2,1}L_{2,3}<\min(I_2,I_1I_3)$.
	Let $\mu_2$, $L_{3,1}$ and $L_{3,2}$ be integers such that $(1+1/\ln(\sqrt{\mu_3}))\sqrt{\mu_3}<L_{3,1},L_{3,2}$ and $L_{3,1}L_{3,2}<\min(I_3,I_1I_2)$.  Let $\sqrt{\mu_1}$, $\sqrt{\mu_2}$ and $\sqrt{\mu_3}$ be positive integers.
	For each $n$, we define $a_{n}$, $a_{n}'$, $c_{nm}$, and $c_{nm}'$ as in Theorems {\rm\ref{SUB:thm4}} and {\rm\ref{SUB:thm5}} with $m=1,2,3$ and $m\neq n$.
	
	For a given tensor $\mathcal{A}\in\mathbb{R}^{I_1\times I_2\times I_3}$, three orthonormal matrices $\mathbf{Q}_n$ are obtained by Algorithm {\rm\ref{SUB:alg3}}. Then
	\begin{equation}\label{SUB:eqn6}
	\left\|\mathcal{A}-\mathcal{A}\times_1({\bf Q}_1{\bf Q}_1^{\top})\times_2({\bf Q}_2{\bf Q}_2^{\top}) \times_3({\bf Q}_3{\bf Q}_3^{\top})\right\|_F\leq 2\sum_{n=1}^3C_n\Delta_{\mu_n+1}(\mathbf{A}_{(n)})
	\end{equation}
	with probability at least
	\begin{equation*}	1-\left(e^{-c_{12}'L_{1,2}}+e^{-c_{13}'L_{1,3}}+e^{-c_{21}'L_{2,1}}+e^{-c_{23}'L_{2,3}}+e^{-c_{31}'L_{3,1}}+e^{-c_{32}'L_{3,2}}+e^{-a_{1}'I_2I_3}+e^{-a_{2}'I_1I_3}+e^{-a_{3}'I_1I_2}\right),
	\end{equation*}
	where  $C_1$, $C_2$ and $C_3$ are given by
	\begin{equation*}
	\begin{split}
	C_1&=\sqrt{\frac{a_{1}^2I_2I_3}{c_{12}^2c_{13}^2L_{1,2}L_{1,3}}+1}
	+\sqrt{\frac{a_{1}^2I_2I_3}{c_{12}^2c_{13}^2L_{1,2}L_{1,3}}},\
	C_2=\sqrt{\frac{a_{2}^2I_1I_3}{c_{21}^2c_{23}^2L_{2,1}L_{2,3}}+1}
	+\sqrt{\frac{a_{2}^2I_1I_3}{c_{21}^2c_{23}^2L_{2,1}L_{2,3}}},\\
	C_3&=\sqrt{\frac{a_{3}^2I_1I_2}{c_{31}^2c_{32}^2L_{3,1}L_{3,2}}+1}
	+\sqrt{\frac{a_{3}^2I_1I_2}{c_{31}^2c_{32}^2L_{3,1}L_{3,2}}}.
	\end{split}
	\end{equation*}
\end{theorem}

\begin{remark}
	We assume that $\sqrt{\mu_1}$, $\sqrt{\mu_2}$ and $\sqrt{\mu_3}$ are positive integers in Theorem {\rm \ref{SUB:thm6}}. In general, we can also consider the case that $\sqrt{\mu_1}$, $\sqrt{\mu_2}$ and $\sqrt{\mu_3}$ are not positive integers.
\end{remark}

Suppose that $\mathbf{A}_{(1)}\in\mathbb{R}^{I_1\times I_2I_3}$ is the mode-1 unfolding of $\mathcal{A}\in\mathbb{R}^{I_1\times I_2\times I_3}$. Let $\mathbf{A}_{(1)}=\mathbf{U}\bm{\Sigma}\mathbf{V}^\top$ be the singular value decomposition of $\mathbf{A}_{(1)}$, where $\mathbf{U}\in\mathbb{R}^{I_1\times I_1}$ and $\mathbf{V}\in\mathbb{R}^{I_2I_3\times I_2I_3}$ are orthogonal and $\bm{\Sigma}\in\mathbb{R}^{I_1\times I_2I_3}$ is diagonal with positive diagonal elements. If $\mathcal{B}=\mathbf{A}\times_1 \mathbf{Q}_1\times_2 \mathbf{Q}_2\times_3 \mathbf{Q}_3$, where $\mathbf{Q}_n\in\mathbb{R}^{I_n\times I_n}$ are orthogonal with $n=1,2,3$, then we have
$$\mathbf{B}_{(1)}=(\mathbf{Q}_1\mathbf{U})\bm{\Sigma}(\mathbf{V}(\mathbf{Q}_3\otimes \mathbf{Q}_2))^\top,$$
where $\mathbf{B}_{(1)}$ is the mode-1 unfolding of $\mathcal{B}$. It implies that the singular values of $\mathbf{B}_{(1)}$ are the same as that of $\mathbf{A}_{(1)}$. Similarly, the singular values of the mode-$n$ unfolding of $\mathcal{A}$ are the same as that of the mode-$n$ unfolding of $\mathcal{B}$ with $n=1,2,3$. Thus, the upper bound in Theorem \ref{SUB:thm6} is orthogonal invariant.

For the case of $n=1$, we set $L_{1,2}L_{1,3}\geq\mu_1+K$ in Algorithm \ref{SUB:alg3} and $\min(I_1,I_2I_3)>L_{1,2}L_{1,3}>(1+1/\ln(\mu_1))\mu_1$ in Theorem \ref{SUB:thm6}. In practical, we set $L_{1,2}L_{1,3}$ is the smallest positive integer such that $L_{1,2}L_{1,3}\geq\mu_1+K$ and $\min(I_1,I_2I_3)>L_{1,2}L_{1,3}>(1+1/\ln(\mu_1))\mu_1$. Let $M=\max(\mu_1+K,(1+1/\ln(\mu_1))\mu_1)$. In practice, we set $L_{1,2}={\rm ceil}(\sqrt{M})$ and $L_{1,3}={\rm round}(\sqrt{M})$, where for $x\in\mathbb{R}$, ${\rm ceil}(x)$ rounds the value of $x$ to the nearest integer towards plus infinity and ${\rm round}(x)$ rounds the value of $x$ to the nearest integer.

In practice, in order to reduce the computational complexity of Algorithm \ref{SUB:alg3}, similar to Algorithm 3.2 in \cite{vvm_2012_sisc}, a slight modification of Algorithm \ref{SUB:alg3} is summarized in Algorithm \ref{SUB:alg8}. Based on (\ref{SUB:eqn5}) and the fact $\|\mathbf{A}\mathbf{Q}\|_F\leq\|\mathbf{A}\|_F$ for $\mathbf{A}\in\mathbb{R}^{I\times J}$ and any  orthonormal matrix $\mathbf{Q}\in\mathbb{R}^{J\times K}\ (K\leq J)$, the temporary tensor $\mathcal{C}$ in Algorithm \ref{SUB:alg8} is updated for each $n$.

\begin{algorithm}[htb]
	\caption{A slight modification of Algorithm \ref{SUB:alg3}}
	\label{SUB:alg8}
	\begin{algorithmic}[1]
		\STATEx {\bf Input}: A tensor $\mathcal{A}\in \mathbb{R}^{I_1\times I_2\times I_3}$ to decompose, the desired multilinear rank $\{\mu_1,\mu_2,\mu_3\}$, $L_{3,1}L_{3,2}\geq\mu_3+K$, $L_{2,1}L_{2,3}\geq\mu_2+K$, $L_{1,2}L_{1,3}\geq\mu_1+K$, and a processing order $\mathbf{p}\in\mathbb{S}_3$, where $K$ is a oversampling parameter.
		\STATEx {\bf Output}: Three orthonormal matrices $\mathbf{Q}_n$ such that $\|\mathcal{A}\times_1 (\mathbf{Q}_1\mathbf{Q}_1^\top)\times_2 (\mathbf{Q}_2\mathbf{Q}_2^\top)
		\times_3 (\mathbf{Q}_3\mathbf{Q}_3^\top)-\mathcal{A}\|_F\leq
		\sum_{n=1}^3O(\Delta_{\mu_n+1}(\mathbf{A}_{(n)}))$.
		\STATE Set  the temporary tensor: $\mathcal{C}=\mathcal{A}$.
		\FOR {$n=p_1,p_2,p_3$}
		\STATE Form two real matrices $\mathbf{G}_{n,m}\in\mathbb{R}^{L_{n,m}\times I_m}$ whose entries are i.i.d. Gaussian random variables of zero mean and unit variance, where $m=1,2,3$ and $m\neq n$.
		\STATE Compute the product tensor
		\begin{equation*}
		\mathcal{B}_n=\mathcal{C}\times_1\mathbf{G}_{n,1}\dots\times_{m-1}\mathbf{G}_{n,m-1}
		\times_{m+1}\mathbf{G}_{n,m+1}\dots\times_3\mathbf{G}_{n,3}.
		\end{equation*}
		\STATE Form the mode-$n$ unfolding $\mathbf{B}_{n,(n)}$ of the tensor $\mathcal{B}_n$.
		\STATE For the $\mathbf{B}_{n,(n)}$, find a real $I_n\times \mu_n$ matrix $\mathbf{Q}_n$ whose columns are orthonormal, such that there exists a real $\mu_n\times \prod_{m=1,m\neq n}^3L_{n,m}$ matrix $\mathbf{S}_n$ for which
		$$\|\mathbf{Q}_n\mathbf{S}_n-\mathbf{B}_{n,(n)}\|_2\leq\sigma_{\mu_n+1}(\mathbf{B}_{n,(n)}).$$
		\STATE Set $I_n=\mu_n$ and $\mathbf{Q}_n=\mathbf{Q}_n(:,1:\mu_n)$, and compute $\mathcal{C}=\mathcal{C}\times_n\mathbf{Q}_n^\top$.
		\ENDFOR
	\end{algorithmic}
\end{algorithm}
\begin{remark}
	Note that $\mathbb{S}_3$ is the Nth order symmetric group on the set $\{1, 2, 3\}$. Since the cardinality of $\mathbb{S}_3$ is $6$, choosing an optimal processing order is an open problem. In practice, the processing order is chosen with $I_{p_1}\geq I_{p_2}\geq I_{p_3}$.
\end{remark}

\subsection{Computational complexity analysis}

In this paper, for clarity, we assume that $I_1= I_2= I_3= I$, $\mu_1=\mu_2=\mu_3=\mu$ and $L_{n,1}= L_{n,2}= L_{n,3}= L$ with $m=1,2,3$ in complexity estimates\footnote{We can also assume that $I_1\sim I_2\sim I_3\sim I$, $\mu_1\sim\mu_2\sim\mu_3\sim\mu$ and $L_{n,1}\sim L_{n,2}\sim L_{n,3}\sim L$ in complexity estimates \cite[Page A2]{GOS_SISC_2012}, where $I_n\sim I$ means $I_n=\alpha_n I$ for some constant $\alpha_n$.}.

To compute the number of floating points operations in Algorithm \ref{SUB:alg3}, we evaluate the complexity of each step:
\begin{enumerate}
	\item[(a)] Generating six standard Gaussian matrices requires $6IL$ operations.
	\item[(b)] Computing three product tensors $\mathcal{B}_{n}\ (n=1,2,3)$ needs $6(L I^3+L^2 I^2)$ operations for the tensor $\mathcal{A}$.
	\item[(c)] Forming the mode-$n$ unfolding $\mathbf{B}_{n,(n)}$ requires $O(IL^2)$ operations.
	\item[(d)] Computing $\mathbf{Q}_n$ requires $O(IL^4)$ operations with $n=1,2,3$.
	\item[(e)] For each $n$, selecting the first $\mu$ columns (we do not modify them) requires $O(1)$ operations.
\end{enumerate}
By summing up the complexities of all the steps above, then Algorithm \ref{SUB:alg3} necessitates
\begin{equation*}
6(IL+L I^3+L^2 I^2)+O(IL^2+IL^4)
\end{equation*}
operations for tensor $\mathcal{A}$.

In order to compute the number of floating points operations in Algorithm \ref{SUB:alg8}, we set $p_1=1$, $p_2=2$ and $p_3=3$.

For the case of $n=1$, generating two standard Gaussian matrices requires $2IL$ operations, computing the product tensor $\mathcal{B}_{1}$ needs $2(I^3L+I^2L^2)$ operations and computing $\mathcal{C}$ requires $2I^3\mu$ operations. For the case of $n=2$, generating two standard Gaussian matrices requires $I(L+\mu)$ operations, computing the product tensor $\mathcal{B}_{1}$ needs $2(LI^2\mu+I^2L^2)$ operations and computing $\mathcal{C}$ requires $2I^2\mu^2$ operations. For the case of $n=3$, generating two standard Gaussian matrices requires $2\mu L$ operations and computing the product tensor $\mathcal{B}_{1}$ needs $2(LI\mu^2+IL^2\mu)$ operations.

Note that for each $n$, the number of entries of $\mathcal{B}_n$ in Algorithm \ref{SUB:alg8} is $IL^2$, then for each $n$, we have
\begin{enumerate}[(i)]
	\item forming the mode-$n$ unfolding $\mathbf{B}_{n,(n)}$ requires $O(IL^2)$ operations;
	\item computing $\mathbf{Q}_n$ requires $O(IL^4)$ operations;
	\item selecting the first $\mu$ columns (we do not modify them) requires $O(1)$ operations.
\end{enumerate}

By summing up the complexities of all the steps above, then Algorithm \ref{SUB:alg8} necessitates
\begin{equation*}
\begin{split}
&2(LI\mu^2+IL^2\mu+2LI^2\mu+I^2\mu^2+\mu I^3+2I^2L^2+L I^3)\\
&+3I(L+\mu)+O(IL^2+IL^4)
\end{split}
\end{equation*}
operations for tensor $\mathcal{A}$.

Note that the main difference between Algorithms \ref{SUB:alg3} and \ref{SUB:alg8} is that the temporary tensor $\mathcal{C}$ are updated after each $n$. We illustrate the difference via an example. The test tensor is defined as $\mathcal{A}={\rm sptenrand}([400,400,400],8000)\in\mathbb{R}^{400\times 400\times 400}$, where ${\rm sptenrand}([400,400,400],8000)$ creates a random sparse tensor in $\mathbb{R}^{400\times 400\times 400}$ with approximately $8000$ nonzero entries \cite{tensortool}. Figure \ref{SUB:fig3add} shows that Algorithm \ref{SUB:alg8} is more efficient than Algorithm \ref{SUB:alg3} for computing low multilinear rank approximations. In the following, Algorithm \ref{SUB:alg8} is denoted as Tucker-SVD.
\begin{figure}[htb]
	\centering
	\includegraphics[width=7in, height=2.5in]{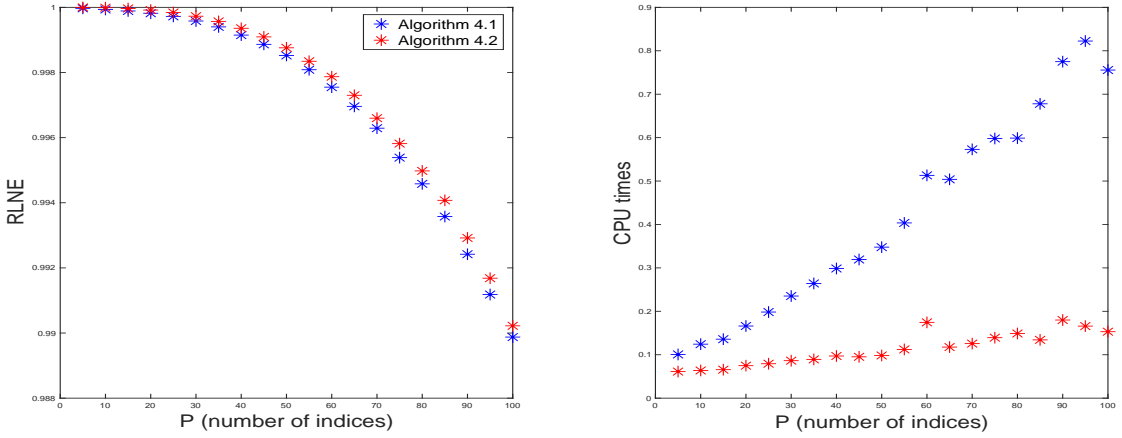}\\
	\caption{Numerical simulation results of applying Algorithms \ref{SUB:alg3} and \ref{SUB:alg8} to  tensor $\mathcal{A}$ with $P=5,10,\dots,100$ and $I=400$. Note that RLNE in the left part is defined in (\ref{SUB:eqn21}).}\label{SUB:fig3add}
\end{figure}
\subsection{Comparison with the existing randomized algorithms}
\label{SUB:sect41}

Suppose that the multilinear rank of $\mathcal{A}\in\mathbb{R}^{I_1\times I_2\times I_3}$ is given as $\{\mu_1,\mu_2,\mu_3\}$, then Algorithm 3.2 in  \cite{che2018randomized} can be represented as follows:
\begin{algorithmic}[1]
	\STATE Set $L_1'\geq\mu_1+K$, $L_2'\geq\mu_2+K$ and $L_3'\geq\mu_3+K$, where $K$ is an oversampling parameter.
	\STATE Set the temporary tensor: $\mathcal{C}=\mathcal{A}$.
	\FOR {$n=p_1,p_2,p_3$}
	\STATE Compute ${\bf B}_{n,(n)}={\bf C}_{(n)}{\bf\Omega}_{(n)}$, where ${\bf \Omega}_{(n)}={\bf \Omega}_{1}'\odot\dots\odot{\bf \Omega}_{n-1}'\odot{\bf\Omega}_{n+1}'\odot\dots\odot{\bf\Omega}_{3}'$ and ${\bf\Omega}_{m}'\in\mathbb{R}^{I_m\times L_m'}$ is a standard Gaussian matrix with $m\neq n$ and $m=1,2,3$.
	\STATE Compute ${\bf Q}_n$ as an orthonormal basis of ${\bf Z}_{(n)}$ by using the QR decomposition and let ${\bf Q}_n={\bf Q}_n(:,1:\mu_n)$.
	\STATE Set $\mathcal{C}=\mathcal{C}\times {\bf Q}_n^{\top}$ and let $I_n=\mu_n$.
	\ENDFOR
\end{algorithmic}

We also list the Randomized Tucker decomposition \cite[Algorithm 2]{decomposition_big_tensor} as follows:
\begin{algorithmic}[1]
	\STATE Set $L_1'\geq\mu_1+K$, $L_2'\geq\mu_2+K$ and $L_3'\geq\mu_3+K$, where $K$ is an oversampling parameter.
	\STATE Set  the temporary tensor: $\mathcal{C}=\mathcal{A}$.
	\FOR {$n=p_1,p_2,p_3$}
	\STATE Compute ${\bf B}_{n,(n)}={\bf C}_{(n)}{\bf\Omega}_{(n)}$, where ${\bf \Omega}_{(n)}$ is an $(\prod_{k\neq n}^3I_k)$-by-$L_n'$ standard Gaussian matrix.
	\STATE Compute ${\bf Q}_n$ as an orthonormal basis of ${\bf Z}_{(n)}$ by using the QR decomposition and let ${\bf Q}_n={\bf Q}_n(:,1:\mu_n)$.
	\STATE Set $\mathcal{C}=\mathcal{C}\times {\bf Q}_n^{\top}$ and let $I_n=\mu_n$.
	\ENDFOR
\end{algorithmic}

Algorithm \ref{SUB:alg8} can be rewritten as follows:
\begin{algorithmic}[1]
	\STATE Set $L_{1,2}L_{1,3}\geq\mu_1+K$, $L_{2,1}L_{2,3}\geq\mu_2+K$ and $L_{3,1}L_{3,2}\geq\mu_3+K$, where $K$ is an oversampling parameter.
	\STATE Set  the temporary tensor: $\mathcal{C}=\mathcal{A}$.
	\FOR {$n=p_1,p_2,p_3$}
	\STATE Compute ${\bf B}_{n,(n)}={\bf C}_{(n)}{\bf\Omega}_{(n)}$, where ${\bf \Omega}_{(n)}={\bf \Omega}_{1}'\otimes\dots\otimes{\bf \Omega}_{n-1}'\times{\bf\Omega}_{n+1}'\otimes\dots\otimes{\bf\Omega}_{3}'$ and ${\bf\Omega}_{n}'\in\mathbb{R}^{L_{n,m}\times I_m}$ is a standard Gaussian matrix with $m\neq n$ and $m=1,2,3$.
	\STATE Compute ${\bf Q}_n$ as an orthonormal basis of ${\bf Z}_{(n)}$ by using singular value decomposition and let ${\bf Q}_n={\bf Q}_n(:,1:\mu_n)$.
	\STATE Set $\mathcal{C}=\mathcal{C}\times {\bf Q}_n^{\top}$ and let $I_n=\mu_n$.
	\ENDFOR
\end{algorithmic}

The main difference among Algorithm \ref{SUB:alg8}, Algorithm 3.2 in  \cite{che2018randomized} and Algorithm 2 in \cite{decomposition_big_tensor} is how to generate the matrix ${\bf B}_{n,(n)}$ for each $n$. For all $n$, generating six standard Gaussian matrices requires $3I(L+\mu)$ operations for Algorithm \ref{SUB:alg8}, $3I(L'+\mu)$ operations for Algorithm 3.2 in  \cite{che2018randomized} and $I^2L'+IL'\mu+L'\mu^2$ for Algorithm 2 in \cite{decomposition_big_tensor}, where we assume that $L_1'=L_2'=L_3'=L'>L$.

\section{Proof for main theorems}\label{SUB:sect5}

In this section, we provide the proof for our main theorem.

\subsection{Some lemmas}

In this section, we obtain some prerequisite results for proving Theorem \ref{SUB:thm6}.

\begin{lemma}\label{SUB:lemma1add}
	Let $I$, $J$ and $K$ be three positive integers such that $K<J<I$. Suppose that $\mathbf{Q}\in\mathbb{R}^{I\times K}$ is orthonormal. For a given $\mathbf{A}\in\mathbb{R}^{I\times J}$, we have
	$$\sigma_{\max}(\mathbf{Q}^\top\mathbf{A})\leq\sigma_{\max}(\mathbf{A}),\quad
	\sigma_{\min}(\mathbf{Q}^\top\mathbf{A})\geq\sigma_{\min}(\mathbf{A}).$$
\end{lemma}
\begin{proof}
	The proof is straightforward, but tedious, as follows. By the definition of singular values of matrices, we have
	\begin{equation*}
	\begin{split}
	&\sigma_{\min}(\mathbf{Q}^\top\mathbf{A})
	=\min_{\mathbf{u}\in\mathbb{R}^{K},\mathbf{u}\neq\mathbf{0}_K;\atop
		\mathbf{v}\in\mathbb{R}^{J},\mathbf{v}\neq\mathbf{0}_J}
	\frac{\mathbf{u}^\top\mathbf{Q}^\top\mathbf{A}\mathbf{v}}
	{\|\mathbf{u}\|_2\|\mathbf{v}\|_2}\\
	&=\min_{\mathbf{u}\in\mathbb{R}^{K},\mathbf{u}\neq\mathbf{0}_K;\atop
		\mathbf{v}\in\mathbb{R}^{J},\mathbf{v}\neq\mathbf{0}_J}
	\frac{(\mathbf{Q}\mathbf{u})^\top\mathbf{A}\mathbf{v}}
	{\|\mathbf{u}\|_2\|\mathbf{v}\|_2}=\min_{\mathbf{u}\in\mathbb{R}^{K},\mathbf{u}\neq\mathbf{0}_K;\atop
		\mathbf{v}\in\mathbb{R}^{J},\mathbf{v}\neq\mathbf{0}_J}
	\frac{(\mathbf{Q}\mathbf{u})^\top\mathbf{A}\mathbf{v}}
	{\|\mathbf{Q}\mathbf{u}\|_2\|\mathbf{v}\|_2}\\
	&\geq
	\min_{\widetilde{\mathbf{u}}\in\mathbb{R}^{I},\widetilde{\mathbf{u}}\neq\mathbf{0}_I;\atop
		\widetilde{\mathbf{v}}\in\mathbb{R}^{J},\widetilde{\mathbf{v}}\neq\mathbf{0}_J}
	\frac{\widetilde{\mathbf{u}}^\top\mathbf{A}\widetilde{\mathbf{v}}}
	{\|\widetilde{\mathbf{u}}\|_2\|\widetilde{\mathbf{v}}\|_2}=\sigma_{\min}(\mathbf{A}).
	\end{split}
	\end{equation*}
	The third equality holds for the fact that $\|\mathbf{Q}\mathbf{u}\|_2=\|\mathbf{u}\|_2$ and the inequality holds for the basic results of optimization theory. Similarly, we can prove $\sigma_{\max}(\mathbf{Q}^\top\mathbf{A})\leq\sigma_{\max}(\mathbf{A})$.
\end{proof}

For two given $\mathbf{A}\in\mathbb{R}^{I\times J}$ and $\mathbf{G}\in\mathbb{R}^{J\times K}$, the following lemma states the singular value of the product $\mathbf{A}\mathbf{G}$ are at most $\|\mathbf{G}\|_2$ times greater than the corresponding singular values of $\mathbf{A}$.

\begin{lemma}{{\bf (\cite[Lemma 3.9]{woolfe2008a})}}
	\label{SUB:lem5}
	Suppose that $\mathbf{A}\in\mathbb{R}^{I\times J}$ and $\mathbf{G}\in\mathbb{R}^{J\times K}$. Then for all $k=1,2,\dots,\min\{I,J,K\}-1,\min\{I,J,K\}$, the $k$th greatest singular value $\sigma_k(\mathbf{A}\mathbf{G})$ of $\mathbf{A}\mathbf{G}$ is at most a factor of $\|\mathbf{G}\|_2$ times greater than the $k$th greatest singular value $\sigma_k(\mathbf{A})$ of $\mathbf{A}$, that is,
	\begin{equation*}
	\sigma_k(\mathbf{A}\mathbf{G})\leq\|\mathbf{G}\|_2\sigma_k(\mathbf{A}).
	\end{equation*}
\end{lemma}

Similar to Lemma \ref{SUB:lem5}, we have the following corollary.
\begin{corollary}
	\label{SUB:lem6}
	Suppose that $\mathbf{A}\in\mathbb{R}^{I\times J}$ and $\mathbf{G}\in\mathbb{R}^{J\times K}$ with $K\leq\min(I,J)$. Then for all $k=1,2,\dots,\min(I,J,K)-1,\min(I,J,K)$, we have
	\begin{equation*}
	\sum_{i=k}^K\sigma_i(\mathbf{A}\mathbf{G})^2
	\leq\|\mathbf{G}\|_2^2\sum_{j=k}^{\min(I,J)}\sigma_j(\mathbf{A})^2.
	\end{equation*}
\end{corollary}

The following classical lemma provides an approximation $\mathbf{Q}\mathbf{S}$ to  $\mathbf{A}\in\mathbb{R}^{I\times J}$ via an  orthonormal matrix $\mathbf{Q}\in\mathbb{R}^{I\times K}$ and $\mathbf{S}\in\mathbb{R}^{K\times J}$.

\begin{lemma}
	\label{SUB:lem10}
	Suppose that $K$, $I$ and $J$ are positive integers with $K< J$ and $J\leq I$. Let $\mathbf{A}\in\mathbb{R}^{I\times J}$. Then there exist an orthonormal matrix $\mathbf{Q}\in\mathbb{R}^{I\times K}$ and $\mathbf{S}\in\mathbb{R}^{K\times J}$ such that
	$$\|\mathbf{Q}\mathbf{S}-\mathbf{A}\|_F\leq\Delta_{K+1}(\mathbf{A}),$$
	with $\Delta_{K+1}(\mathbf{A}):=(\sum_{i=K+1}^J\sigma_i(\mathbf{A})^2)^{1/2}$, where $\sigma_{i}(\mathbf{A})$ is the $i$th greatest singular value of $\mathbf{A}$ for all $i=1,2,\dots, J$.
\end{lemma}
\begin{proof}
	The proof is similar to that of Lemma 3.5 in \cite{martinsson2011a}. We start by form an SVD of $\mathbf{A}$
	\begin{equation*}
	\mathbf{A}=\mathbf{U}\bm{\Sigma}\mathbf{V}^\top,
	\end{equation*}
	where $\mathbf{U}\in\mathbb{R}^{I\times J}$ is orthonormal, $\mathbf{V}\in\mathbb{R}^{J\times J}$ is orthogonal, and $\bm{\Sigma}\in^{I\times J}$ is diagonal with nonnegative diagonal entries. Let $\mathbf{Q}=\mathbf{U}(:,1:K)$ and $\mathbf{S}=\bm{\Sigma}(1:K,1:K)\mathbf{V}(:,1:K)^\top$. Note that $\mathbf{A}_K=\mathbf{U}(:;1:K)\bm{\Sigma}(1:K,1:K)\mathbf{V}(:,1:K)^\top$ is a best rank-$K$ approximation of $\mathbf{A}$. Then we have
	$$\|\mathbf{Q}\mathbf{S}-\mathbf{A}\|_F=\|\mathbf{A}_K-\mathbf{A}\|_F\leq\Delta_{K+1}(\mathbf{A}),$$
	which implies this lemma.
\end{proof}
\begin{remark}
	\label{SUB:rem2}
	In order to compute matrices $\mathbf{Q}$ and $\mathbf{S}$ in Lemma {\rm\ref{SUB:lem10}} from matrix $\mathbf{A}$, we can construct the SVD of $\mathbf{A}$, and then form $\mathbf{Q}$ and $\mathbf{S}$ from this decomposition. For example, details concerning the computation of the SVD can be found in  {\rm \cite[Chapter 8]{Golub}}.
\end{remark}

Without loss of generality, we assume that $n=1$. The following lemma states that the product $\mathcal{A}\times_1(\mathbf{Q}_1\mathbf{Q}_1^\top)$ of $\mathcal{A}$, $\mathbf{Q}_1$ and $\mathbf{Q}_1^\top$ is a good approximation to $\mathcal{A}$, provided that there exist matrices $\mathbf{G}_{1,m}\in\mathbb{R}^{L_{1,m}\times I_m}\ (m=2,3)$ and $\mathbf{S}_1\in\mathbb{R}^{\mu_1\times L_{1,2}L_{1,3}}$ such that
(a) $\mathbf{Q}_1$ is orthonormal; (b) $\mathbf{Q}_1\mathbf{S}_1$ is a good approximation to $(\mathcal{A}\times_2\mathbf{G}_{1,2}\times_3\mathbf{G}_{1,3})_{(1)}$; (c) there exist a matrix $\mathbf{F}\in\mathbb{R}^{L_{1,2}L_{1,3}\times I_2I_3}$ such that $\|\mathbf{F}\|_2$ is not too large, and $\mathcal{A}_{(1)}(\mathbf{G}_{1,3}\otimes\mathbf{G}_{1,2})^\top\mathbf{F}$ is a good approximation to $\mathcal{A}_{(1)}$.

\begin{lemma}
	\label{SUB:lem11}
	Suppose that $\mathcal{A}\in\mathbb{R}^{I_1\times I_2\times I_3}$, $\mathbf{Q}_1\in\mathbb{R}^{I_1\times \mu_1}$ is orthonormal with $\mu_1\leq I_1$, $\mathbf{S}_1$ is a real $\mu_1\times L_{1,2}L_{1,3}$ matrix, $\mathbf{F}$ is a real $L_{1,2}L_{1,3}\times I_2I_3$ matrix, and $\mathbf{G}_{1,m}$ is a real $L_{1,m}\times I_m$ matrix with $m=2,3$. Then
	\begin{equation}\label{SUB:eqn11}
	\begin{split}
	\left\|\mathcal{A}-\mathcal{A}\times_1(\mathbf{Q}_1\mathbf{Q}_1^\top)\right\|_F^2
	&\leq2\left\|\mathbf{A}_{(1)}(\mathbf{G}_{1,3}\otimes\mathbf{G}_{1,2})^\top\mathbf{F}
	-\mathbf{A}_{(1)}\right\|_F^2\\
	&+2\|\mathbf{F}\|_2^2\left\|\mathcal{S}_1\times_1\mathbf{Q}_1
	-\mathcal{A}\times_2\mathbf{G}_{1,2}\times_3\mathbf{G}_{1,3}\right\|_F^2,
	\end{split}
	\end{equation}
	where the entries of $\mathcal{S}_1\in\mathbb{R}^{\mu_1\times L_{1,2}\times L_{1,3}}$ are given by $\mathcal{S}_1(i_1,i_2,i_3)=s_{ij}$, with $i=i_1$ and $j=i_2+(i_3-1)L_{1,2}$
	for all $i_1=1,2,\dots,\mu_1$, $i_2=1,2,\dots,L_{1,2}$ and $i_3=1,2,\dots,L_{1,3}$.
\end{lemma}

\begin{proof}
	The proof is straightforward, but tedious, as follows. By using the triangular inequality, we have
	\begin{equation}\label{SUB:eqn8}
	\begin{split}
	\left\|\mathcal{A}-\mathcal{A}\times_1(\mathbf{Q}_1\mathbf{Q}_1^\top)\right\|_F^2&\leq\left\|(\mathbf{Q}_1\mathbf{Q}_1^\top)\mathbf{A}_{(1)}(\mathbf{G}_{1,3}\otimes\mathbf{G}_{1,2})^\top\mathbf{F}
	-(\mathbf{Q}_1\mathbf{Q}_1^\top)\mathbf{A}_{(1)}\right\|_F^2\\
	&+\left\|(\mathbf{Q}_1\mathbf{Q}_1^\top)\mathbf{A}_{(1)}(\mathbf{G}_{1,3}\otimes\mathbf{G}_{1,2})^\top\mathbf{F}
	-\mathbf{A}_{(1)}(\mathbf{G}_{1,3}\otimes\mathbf{G}_{1,2})^\top\mathbf{F}\right\|_F^2\\
	&+\left\|\mathbf{A}_{(1)}(\mathbf{G}_{1,3}\otimes\mathbf{G}_{1,2})^\top\mathbf{F}
	-\mathbf{A}_{(1)}\right\|_F^2.
	\end{split}
	\end{equation}
	For the first term in the right-hand side of (\ref{SUB:eqn8}), we have
	\begin{equation*}
	\begin{split}
	\left\|(\mathbf{Q}_1\mathbf{Q}_1^\top)\mathbf{A}_{(1)}(\mathbf{G}_{1,3}\otimes\mathbf{G}_{1,2})^\top\mathbf{F}
	-(\mathbf{Q}_1\mathbf{Q}_1^\top)\mathbf{A}_{(1)}\right\|_F^2\leq\left\|\mathbf{A}_{(1)}(\mathbf{G}_{1,3}\otimes\mathbf{G}_{1,2})^\top\mathbf{F}
	-\mathbf{A}_{(1)}\right\|_F^2\|\mathbf{Q}_1\mathbf{Q}_1^\top\|_2^2.
	\end{split}
	\end{equation*}
	Since $\|\mathbf{Q}_1\mathbf{Q}_1^\top\|_2\leq1$, then
	\begin{equation}\label{SUB:eqn9}
	\begin{split}
	\left\|(\mathbf{Q}_1\mathbf{Q}_1^\top)\mathbf{A}_{(1)}(\mathbf{G}_{1,3}\otimes\mathbf{G}_{1,2})^\top\mathbf{F}
	-(\mathbf{Q}_1\mathbf{Q}_1^\top)\mathbf{A}_{(1)}\right\|_F^2\leq\left\|\mathbf{A}_{(1)}(\mathbf{G}_{1,3}\otimes\mathbf{G}_{1,2})^\top\mathbf{F}
	-\mathbf{A}_{(1)}\right\|_F^2.
	\end{split}
	\end{equation}
	Now, we provide a bound for the second term in the right-hand side of (\ref{SUB:eqn8}). Clearly, we have
	\begin{equation*}
	\begin{split}
	&\left\|(\mathbf{Q}_1\mathbf{Q}_1^\top)\mathbf{A}_{(1)}(\mathbf{G}_{1,3}\otimes\mathbf{G}_{1,2})^\top\mathbf{F}
	-\mathbf{A}_{(1)}(\mathbf{G}_{1,3}\otimes\mathbf{G}_{1,2})^\top\mathbf{F}\right\|_F^2\\
	&\leq\left\|(\mathbf{Q}_1\mathbf{Q}_1^\top)\mathbf{A}_{(1)}(\mathbf{G}_3\otimes\mathbf{G}_2)^\top
	-\mathbf{A}_{(1)}(\mathbf{G}_{1,3}\otimes\mathbf{G}_{1,2})^\top\right\|_F^2\|\mathbf{F}\|_2^2.
	\end{split}
	\end{equation*}
	It follows from the triangular inequality that
	\begin{equation*}
	\begin{split}
	&\left\|(\mathbf{Q}_1\mathbf{Q}_1^\top)\mathbf{A}_{(1)}(\mathbf{G}_{1,3}\otimes
\mathbf{G}_{1,2})^\top
-\mathbf{A}_{(1)}(\mathbf{G}_{1,3}\otimes\mathbf{G}_{1,2})^\top\right\|_F^2\\
&\leq\left\|(\mathbf{Q}_1\mathbf{Q}_1^\top)\mathbf{A}_{(1)}(\mathbf{G}_{1,3}\otimes\mathbf{G}_{1,2})^\top
	-\mathbf{Q}_1\mathbf{Q}_1^\top\mathbf{Q}_1\mathbf{S}_1\right\|_F^2\\
	&+\left\|\mathbf{Q}_1\mathbf{Q}_1^\top\mathbf{Q}_1\mathbf{S}_1
	-\mathbf{Q}_1\mathbf{S}_1\right\|_F^2
	+\left\|\mathbf{Q}_1\mathbf{S}_1-\mathbf{A}_{(1)}(\mathbf{G}_{1,3}\otimes\mathbf{G}_{1,2})^\top\right\|_F^2.
	\end{split}
	\end{equation*}
	Since $\mathbf{Q}_1^\top\mathbf{Q}_1=\mathbf{I}_{\mu_1}$, then
	$$\left\|(\mathbf{Q}_1\mathbf{Q}_1^\top)\mathbf{Q}_1\mathbf{S}_1
	-\mathbf{Q}_1\mathbf{S}_1\right\|_F^2=0.$$
	Since $\|\mathbf{Q}_1\mathbf{Q}_1^\top\|_2 =1$, then
	\begin{equation*}
	\begin{split}
	\left\|(\mathbf{Q}_1\mathbf{Q}_1^\top)\mathbf{A}_{(1)}(\mathbf{G}_{1,3}\otimes\mathbf{G}_{1,2})^\top
	-\mathbf{Q}_1\mathbf{Q}_1^\top\mathbf{Q}_1\mathbf{S}_1\right\|_F^2\leq
	\left\|\mathbf{Q}_1\mathbf{S}_1-\mathbf{A}_{(1)}(\mathbf{G}_{1,3}\otimes\mathbf{G}_{1,2})^\top\right\|_F^2.
	\end{split}
	\end{equation*}
	Hence we have
	\begin{equation}\label{SUB:eqn10}
	\begin{split}
	&\left\|(\mathbf{Q}_1\mathbf{Q}_1^\top)\mathbf{A}_{(1)}(\mathbf{G}_{1,3}\otimes\mathbf{G}_{1,2})^\top\mathbf{F}
	-\mathbf{A}_{(1)}(\mathbf{G}_{1,3}\otimes\mathbf{G}_{1,2})^\top\mathbf{F}\right\|_F^2\\
	&\leq 2
	\|\mathbf{F}\|_2^2
	\left\|\mathbf{Q}_1\mathbf{S}_1-\mathcal{A}_{(1)}(\mathbf{G}_{1,3}\otimes\mathbf{G}_{1,2})^\top\right\|_F^2.
	\end{split}
	\end{equation}
	Combining (\ref{SUB:eqn8}), (\ref{SUB:eqn9}) and (\ref{SUB:eqn10}) yields (\ref{SUB:eqn11}).
\end{proof}

The upper bound of (\ref{SUB:eqn11}) is given in the following theorem.
\begin{theorem}
	\label{SUB:thm9}
	Suppose that $\mathcal{A}\in\mathbb{R}^{I_1\times I_2\times I_3}$. Let $\mathbf{G}_{1,m}$ be a real $L_{1,m}\times I_m$ matrix whose entries are i.i.d. Gaussian random variables with zero mean and unit variance for $m=2,3$. Let $\mu_1$, $L_{1,2}$ and $L_{1,3}$ be integers such that $(1+1/\ln(\sqrt{\mu_1}))\sqrt{\mu_1}<L_{1,2},L_{1,3}$ and $L_{1,2}L_{1,3}<\min(I_1,I_2I_3)$. Let $\sqrt{\mu_1}$ be a positive integer.
	We define $a_{1}$, $a_{1}'$, $c_{12}$, $c_{12}'$, $c_{13}$ and $c_{13}'$ as in Theorems {\rm\ref{SUB:thm4}} and {\rm\ref{SUB:thm5}}. Then there exists a matrix $\mathbf{F}\in\mathbb{R}^{L_{1,2}L_{1,3}\times I_2I_3}$ such that
	\begin{equation*}
	\left\|\mathbf{A}_{(1)}(\mathbf{G}_{1,3}\otimes\mathbf{G}_{1,2})^\top\mathbf{F}
	-\mathbf{A}_{(1)}\right\|_F
	\leq C_1'\Delta_{\mu_1+1}(\mathbf{A}_{(1)}),
	\end{equation*}
	and
	\begin{equation*}
	\|\mathbf{F}\|_2\leq\frac{1}{c_{1}\sqrt{L_{1,2}L_{1,3}}},\ C_1'=\sqrt{\frac{a_{1}^2I_2I_3}{c_{12}^2c_{13}^2L_{1,2}L_{1,3}}+1}
	\end{equation*}
	with probability at least $1-e^{-c_{12}'L_{1,2}}-e^{-c_{13}'L_{1,3}}-e^{-a_{1}'I_2I_3}$.
\end{theorem}
\begin{proof}
	We begin by applying  SVD of to $\mathbf{A}_{(1)}$ such that $\mathbf{A}_{(1)}=\mathbf{U}\bm{\Sigma}\mathbf{V}^{\top}$, where $\mathbf{U}\in\mathbb{R}^{I_1\times I_1}$ is orthonormal, $\bm{\Sigma}\in\mathbb{R}^{I_1\times I_1}$ is diagonal with nonnegative entries and $\mathbf{V}\in\mathbb{R}^{I_2I_3\times I_1}$ is orthogonal.
	
	Assume that the product of $\mathbf{V}^\top$ and $\mathbf{G}_{1,3}\otimes\mathbf{G}_{1,2}$ is
	\begin{equation*}
	\mathbf{V}^\top(\mathbf{G}_{1,3}\otimes\mathbf{G}_{1,2})=\begin{pmatrix}
	\mathbf{H}\\
	\mathbf{R}
	\end{pmatrix},
	\end{equation*}
	where $\mathbf{H}$ is a $\mu_1\times L_{1,2}L_{1,3}$ matrix and $\mathbf{R}$ is an $(I_1-\mu_1)\times L_{1,2}L_{1,3}$ matrix. Since $\mathbf{G}_{1,3}\otimes\mathbf{G}_{1,2}$ is a sub-Gaussian matrix, and $\mathbf{V}$ is an orthogonal matrix, then $\mathbf{V}^\top(\mathbf{G}_{1,3}\otimes\mathbf{G}_{1,2})$ is also a sub-Gaussian matrix. Therefore, $\mathbf{H}$ and $\mathbf{R}$ are also sub-Gaussian matrices. Define  $\mathbf{F}=\mathbf{P}\mathbf{V}^\top$, where $\mathbf{P}$ is a matrix of size $L_{1,2}L_{1,3}\times I_1$ such that
	\begin{equation*}
	\mathbf{P}=\begin{pmatrix}
	\mathbf{H}^\dag&\mathbf{0}_{L_{1,2}L_{1,3}\times (I_1-\mu_1)}
	\end{pmatrix}.
	\end{equation*}
	Note that $\mathbf{H}=(\mathbf{V}(:,1:\mu_1))^\top(\mathbf{G}_{1,3}\otimes\mathbf{G}_{1,2})$. According to Lemma \ref{SUB:lemma1add} and Theorem \ref{SUB:thm5}, we get
	\begin{equation*}
	\begin{split}
	\|\mathbf{F}\|_2&=\|\mathbf{P}\mathbf{V}^\top\|_2
	=\|\mathbf{H}^\dag\|_2=
	\leq\frac{1}{\sigma_{\min}(\mathbf{H})}=\frac{1}{\sigma_{\min}(\mathbf{G}_{1,2})}\frac{1}{\sigma_{\min}(\mathbf{G}_{1,3})}
	\leq\frac{1}{c_{12}c_{13}\sqrt{L_{1,2}L_{1,3}}}
	\end{split}
	\end{equation*}
	with probability not less than $1-e^{-c_{12}'L_{1,2}}-e^{-c_{13}'L_{1,3}}$.
	
	Now, we can bound $\|\mathbf{A}_{(1)}(\mathbf{G}_{1,3}\otimes\mathbf{G}_{1,2})^\top\mathbf{F}
	-\mathbf{A}_{(1)}\|_F$.
	By using $\mathbf{A}_{(1)}=\mathbf{U}\bm{\Sigma}\mathbf{V}^{\top}$, we obtain
	\begin{equation*}
	\begin{split}
	\mathbf{A}_{(1)}(\mathbf{G}_{1,3}\otimes\mathbf{G}_{1,2})^\top\mathbf{F}
	-\mathbf{A}_{(1)}=\mathbf{U}\bm{\Sigma}\left(
	\begin{pmatrix}
	\mathbf{H}\\
	\mathbf{R}\end{pmatrix}\begin{pmatrix}
	\mathbf{H}^\dag&\mathbf{0}_{L_{1,2}L_{1,3}\times (I_1-\mu_1)}
	\end{pmatrix}-\mathbf{I}_{I_1}\right)\mathbf{V}^{\top}.
	\end{split}
	\end{equation*}
	We define $\bm{\Sigma}_2$ to be the $(I_1-\mu_1)\times (I_1-\mu_1)$ lower-right block of $\bm{\Sigma}$. Then
	\begin{equation*}
	\begin{split}
	\bm{\Sigma}\left(
	\begin{pmatrix}
	\mathbf{H}\\
	\mathbf{R}\end{pmatrix}\begin{pmatrix}
	\mathbf{H}^\dag&\mathbf{0}_{L_{1,2}L_{1,3}\times (I_1-\mu_1)}
	\end{pmatrix}-\mathbf{I}_{I_1}\right)=\bm{\Sigma}
	\begin{pmatrix}
	\mathbf{0}_{\mu_1\times \mu_1}&\mathbf{0}_{\mu_1\times (I_1-\mu_1)}\\
	\mathbf{R}\mathbf{H}^{\dag}&-\mathbf{I}_{I_1-\mu_1}
	\end{pmatrix}=\begin{pmatrix}
	\mathbf{0}_{\mu_1\times \mu_1}&\mathbf{0}_{\mu_1\times (I_1-\mu_1)}\\
	\bm{\Sigma}_2\mathbf{R}\mathbf{H}^{\dag}&-\bm{\Sigma}_2
	\end{pmatrix}.
	\end{split}
	\end{equation*}
	The Frobenius norm of the last term is
	\begin{equation*}
	\left\|\begin{pmatrix}
	\mathbf{0}_{\mu_1\times \mu_1}&\mathbf{0}_{\mu_1\times (I_1-\mu_1)}\\
	\bm{\Sigma}_2\mathbf{R}\mathbf{H}^{\dag}&-\bm{\Sigma}_2
	\end{pmatrix}\right\|_F
	\leq\left\|\bm{\Sigma}_2\mathbf{R}\mathbf{H}^{\dag}\right\|_F
	+\|\bm{\Sigma}_2\|_F.
	\end{equation*}
	Moreover, we have
	\begin{equation*}
	\begin{split}
	\|\bm{\Sigma}_2\mathbf{R}\mathbf{H}^{\dag}\|_F\leq
	\|\mathbf{H}^{\dag}\|_2\|\mathbf{R}\|_2\|\bm{\Sigma}_2\|_F\leq
	\|\mathbf{H}^{\dag}\|_2\|\mathbf{G}_{1,2}\otimes\mathbf{G}_{1,3}\|_2\|\bm{\Sigma}_2\|_F.
	\end{split}
	\end{equation*}
	By Theorem \ref{SUB:thm4}, we know
	$$\|\mathbf{R}\|_2\leq\|\mathbf{G}_{1,2}\otimes\mathbf{G}_{1,3}\|_2\leq a_{1}\sqrt{I_2I_3}$$
	with probability not less than $1-e^{-a_{1}'I_2I_3}$. Hence, this theorem is completely proved.
\end{proof}
\subsection{Proving Theorem \ref{SUB:thm6}}

In this section, we assume that $\mathbf{Q}_1$ in Lemma \ref{SUB:lem11} is derived from Algorithm \ref{SUB:thm5}. The main goal is to estimate the upper bound of $\|\mathcal{A}-\mathcal{A}\times_1(\mathbf{Q}_1\mathbf{Q}_1^\top)\|_F$. As shown in Lemma \ref{SUB:lem11} and Theorem \ref{SUB:thm9}, we only need  to derive an upper bound for the second part in the right-hand side of (\ref{SUB:eqn11}).

For a given $\mathbf{A}\in\mathbb{R}^{I\times J}$, suppose that the entries of $\mathbf{G}\in\mathbb{R}^{J\times L}$ are i.i.d. sub-Gaussian random variables of zero mean and unit variance, the following theorem provides a highly probable upper bound on the singular values of the product $\mathbf{A}\mathbf{G}$ in term of the singular values of $\mathbf{A}$.

\begin{theorem}
	\label{SUB:thm10}
	Let $\mathbf{A}$ be a real $I\times J$ matrix with $I\leq J$. Let $K$ and $L$ be integers such that $K<L<I$. Suppose that $\mu\geq1$, and the entries of $\mathbf{G}\in\mathbb{R}^{J\times K}$ are i.i.d. sub-Gaussian random variables with zero mean and unit variance. We define $a_1$ and $a_2$ as in Theorems {\rm\ref{SUB:thm4}} and {\rm\ref{SUB:thm5}}. Then
	\begin{equation*}
	\Delta_{K+1}(\mathbf{A}\mathbf{G})\leq a_1\sqrt{J}\Delta_{K+1}(\mathbf{A})
	\end{equation*}
	with probability at least $1-e^{-a_2J}$, where $a_1=6\mu\sqrt{a_2+4}$.
\end{theorem}
\begin{proof}
	By Corollary \ref{SUB:lem6}, we have
	\begin{equation*}
	\sum_{i=K+1}^L\sigma_{i}(\mathbf{A}\mathbf{G})^2
	\leq\|\mathbf{G}\|_2^2\sum_{j=K+1}^I\sigma_{j}(\mathbf{A})^2,
	\end{equation*}
	that is,
	$$\Delta_{K+1}(\mathbf{A}\mathbf{G})\leq \|\mathbf{G}\|_2\Delta_{K+1}(\mathbf{A}).$$
	Since the entries of $\mathbf{G}\in\mathbb{R}^{J\times K}$ are i.i.d. sub-Gaussian random variables with zero mean and unit variance, then, according to Theorem \ref{SUB:thm4}, we have $\|\mathbf{G}\|_2\leq a_1\sqrt{J}$ with probability at least $1-e^{-a_2J}$. Hence, the proof is completed.
\end{proof}

\begin{theorem}
	\label{SUB:thm11}
	Suppose that $\mathcal{A}\in\mathbb{R}^{I_1\times I_2\times I_3}$. Let $\mathbf{G}_{1,m}$ be a real $L_{1,m}\times I_m$ matrix whose entries are i.i.d. Gaussian random variables with zero mean and unit variance for $m=2,3$. Let $\mu_1$, $L_2$ and $L_3$ be integers such that $(1+1/\ln(\sqrt{\mu_1}))\sqrt{\mu_1}<L_{1,2},L_{1,3}$ and $L_{1,2}L_{1,3}<\min(I_1,I_2I_3)$.
	We define $a_{1}$, and $a_{1}'$ as in Theorems {\rm\ref{SUB:thm4}} and {\rm\ref{SUB:thm5}}. Then
	\begin{equation*}
	\Delta_{\mu_1+1}(\mathbf{A}_{(1)}(\mathbf{G}_{1,3}\otimes\mathbf{G}_{1,2})^\top)
	\leq a_{1}\sqrt{I_2I_3}\Delta_{\mu_1+1}(\mathbf{A}_{(1)})
	\end{equation*}
	with probability at least $1-e^{-a_{1}'I_2I_3}$, where $a_1=6\alpha\sqrt{a_1'+4}$ for $\alpha\geq0$.
\end{theorem}
\begin{proof}
	Combining Theorems \ref{SUB:thm4} and \ref{SUB:thm10}, we can prove this theorem.
\end{proof}

Combining Theorems \ref{SUB:thm9} and \ref{SUB:thm11}, we obtain the following theorem.
\begin{theorem}
	\label{SUB:thm13}
	Suppose that $I_1\leq I_2I_3$. Let $\mu_1$, $L_{1,2}$ and $L_{1,3}$ be integers such that $L_{1,2}L_{1,3}<\min(I_1,I_2I_3)$ and $(1+1/\ln(\sqrt{\mu_1}))\sqrt{\mu_1}<L_{1,2},L_{1,3}$. Let  $\sqrt{\mu_1}$ be a positive integer. We define $a_{1}$, $a_{1}'$, $c_{12}$, $c_{12}'$, $c_{13}$ and $c_{13}'$ as in Theorems {\rm\ref{SUB:thm4}} and {\rm\ref{SUB:thm5}}. For a given tensor $\mathcal{A}\in\mathbb{R}^{I_1\times I_2\times I_3}$, if $\mathbf{Q}_1$ is derived from Algorithm {\rm\ref{SUB:alg3}} with $n=1$, then
	\begin{equation*}
	\begin{split}
	\left\|\mathcal{A}-\mathcal{A}\times_1(\mathbf{Q}_1\mathbf{Q}_1^\top)\right\|_F\leq
	2\left(\sqrt{\frac{a_{1}^2I_2I_3}{c_{1}^2L_{1,2}L_{1,3}}+1}
	+\sqrt{\frac{a_{1}^2I_2I_3}{c_{12}^2c_{13}^2L_{1,2}L_{1,3}}}\right)
	\Delta_{\mu_1+1}(\mathbf{A}_{(1)})
	\end{split}
	\end{equation*}
	with probability at least $1-e^{-c_{12}'L_{1,2}}-e^{-c_{13}'L_{1,3}}-e^{-a_{1}'I_2I_3}$.
\end{theorem}

Now, we provide a proof for Theorem \ref{SUB:thm6} based on the above discussions.
\begin{proof}
	Theorem \ref{SUB:thm6} is derived from (\ref{SUB:eqn13}) and Theorem \ref{SUB:thm13}.
\end{proof}

\section{Numerical examples}

\label{SUB:sect6}

In this section, the codes are written using MATLAB and the MATLAB Tensor Toolbox \cite{tensortool} and the computations
are implemented on a laptop with Intel Core i5-4200M CPU (2.50GHz) and 8.00GB RAM.  Floating point numbers in each example have four decimal digits. In order to implement all algorithms in this paper, we set $K=10$. We use three functions `ttv', `ttm' and `ttt' in \cite{tensortool} to implement the tensor-vector product, the tensor-matrix product and the tensor-tensor product, respectively.

  We suppose that $I_1=I_2=I_3:=I$, $\mu_1=\mu_2=\mu_3=P$ and $L_{n,1}=L_{n,2}=L_{n,3}:=P+K$ with $n=1,2,3$. Under these assumptions, $\{p_1,p_2,p_3\}$ in Algorithm \ref{SUB:alg8} is set by $\{1,2,3\}$. For a given low multilinear rank approximation $\widehat{\mathcal{A}}=\mathcal{A}\times_{1}({\bf S}_{1}{\bf S}_{1}^\top)\times_{2}({\bf S}_{2}{\bf S}_{2}^\top)\times_{3}({\bf S}_{3}{\bf S}_{3}^\top)$ of $\mathcal{A}\in\mathbb{R}^{I\times I\times I}$, where the matrices ${\bf S}_n\in\mathbb{R}^{I\times \mu}$ are derived form the desired numerical algorithms. The relative least normalized error (RLNE) of the approximation is defined as
\begin{equation}\label{SUB:eqn21}
	{\rm RLNE}=\|\mathcal{A}-\widehat{\mathcal{A}}\|_F/\|\mathcal{A}\|_F.
\end{equation}

In this section, we compare Tucker-SVD with the existing deterministic and randomized algorithms for computing low multilinear rank approximations of a tensor via several examples. These algorithms are given by:
\begin{enumerate}
	\item[$\bullet$] tucker\_ALS: higher-order orthogonal iteration \cite{tensortool} (the maximum number of iterations is set to 50, the order to loop through dimensions is $\{1,2,3\}$, the entries of initial values are i.i.d. standard Gaussian variables and the tolerance on difference in fit is set to 0.0001);
	\item[$\bullet$] mlsvd: truncated multilinear singular value decomposition \cite{vvm_2012_sisc} (the order to loop through dimensions is $\{1,2,3\}$ and a faster but possibly less accurate eigenvalue decomposition is used to compute the factor matrices);
	\item[$\bullet$] lmlra\_aca: low multilinear rank approximation by adaptive cross-approximation \cite{tensor_cur_2010,tensorlab} (the relative singular value tolerance in determining the factor matrices is set to $1e-12$ and the factor matrices are orthonormal);
	\item[$\bullet$] Adap-Tucker: low multilinear rank approximation by the adaptive randomized algorithm \cite{che2018randomized};
	\item[$\bullet$] ran-Tucker: the randomized Tucker decomposition \cite{decomposition_big_tensor};
	\item[$\bullet$] mlsvd\_rsi: truncated multilinear SVD  \cite{vvm_2012_sisc} by  a randomized SVD algorithm based on randomized subspace iteration \cite{random_2011_siamreview} (the oversampling parameter is 10, the number of subspace iterations to be performed is 2 and we remove the parts of the factor matrices and core tensor corresponding due to the oversampling).
\end{enumerate}

\begin{figure}
	\centering
	\includegraphics[width=7in, height=2.5in]{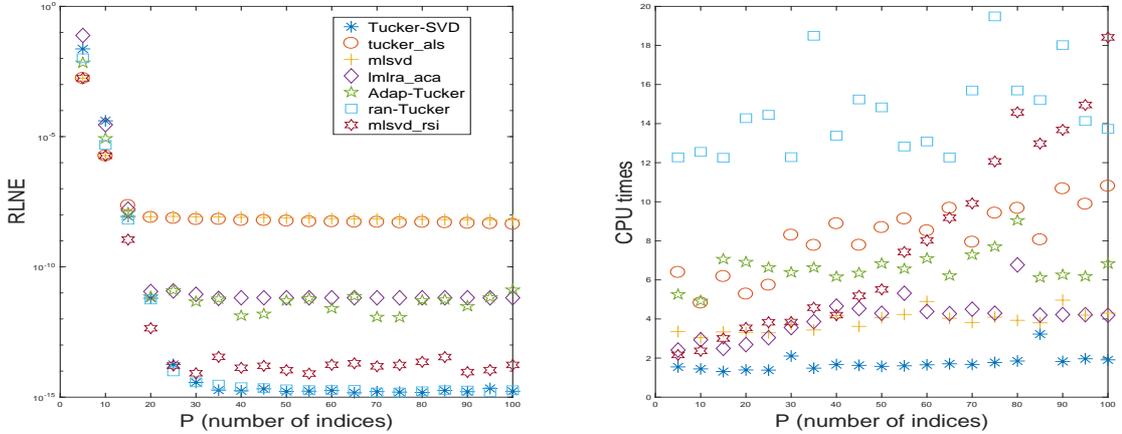}\\
	\caption{Numerical simulation results of applying Tucker-SVD, tucker\_als, mlsvd, lmlra\_aca, Adap-Tucker, ran-Tucker and mlsvd\_rsi to $\mathcal{A}$ with $P=5,10,\dots,100$ and $I=400$.}\label{SUB:fig1}
\end{figure}
\subsection{The test tensors generated by smooth functions}

Now we consider two tensors generated by sampling two families of smooth functions as follows,
\begin{equation*}
	a_{ijk}=\frac{1}{i+j+k},\quad b_{ijk}=\frac{1}{\ln(i+2j+3k)},
\end{equation*}
with $i,j,k=1,2,\dots,I$. The type of tensor $\mathcal{A}$ is chosen from \cite{tensor_cur_2010}.

Suppose that $I=400$. We compute a low multilinear rank approximation of $\mathcal{A}$ and $\mathcal{B}$ with multilinear rank $\{P,P,P\}$ using Tucker-SVD,  tucker\_als, mlsvd, lmlra\_aca, Adap-Tucker, ran-Tucker and mlsvd\_rsi. Figures \ref{SUB:fig1} and \ref{SUB:fig2} compare efficiency and accuracy of different methods on $\mathcal{A}$ and $\mathcal{B}$, respectively.

\begin{figure}
	\centering
	\includegraphics[width=7in, height=2.5in]{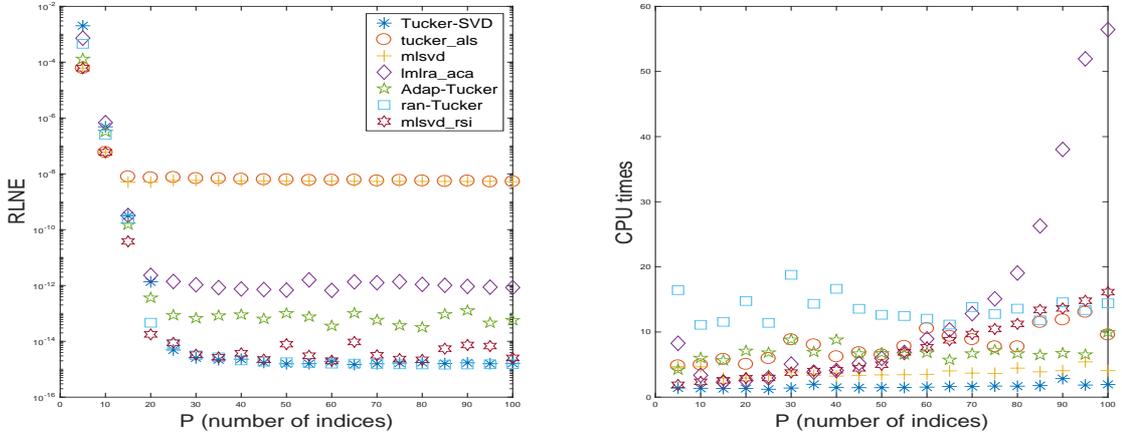}\\
	\caption{Numerical simulation results of applying Tucker-SVD, mlsvd, lmlra\_aca, Adap-Tucker, ran-Tucker and mlsvd\_rsi to $\mathcal{B}$ with $P=5,10,\dots,100$ and $I=400$.}\label{SUB:fig2}
\end{figure}

\subsection{A sparse tensor}

A sparse tensor $\mathcal{A}\in\mathbb{R}^{I\times I\times I}$ is defined as \cite{s_2016_simax,sorensen2016a}
\begin{equation*}
	\mathcal{A}=\sum_{j=1}^{10}\frac{1000}{j}(\mathbf{x}_j
	\circ\mathbf{y}_j\circ\mathbf{z}_j)+
	\sum_{j=11}^I\frac{1}{j}(\mathbf{x}_j
	\circ\mathbf{y}_j\circ\mathbf{z}_j)
\end{equation*}
where $\mathbf{x}_j,\mathbf{y}_j,\mathbf{z}_j\in\mathbb{R}^I$ are sparse vectors with nonnegative entries in MATLAB,
\begin{equation*}
\begin{split}
&\mathbf{x}_j={\rm sprand(I,1,0.015)},\quad\mathbf{y}_j={\rm sprand(I,1,0.025)},\quad \mathbf{z}_j={\rm sprand(I,1,0.035)}.
\end{split}
\end{equation*}
The symbol `$\circ$' represents the vector outer product. Here we assume that $I=400$.
Figure \ref{SUB:fig3} shows the results of RLNE and CPU time for
Tucker-SVD, tucker\_als, mlsvd, lmlra\_aca, Adap-Tucker, ran-Tucker and mlsvd\_rsi used to find a low multilinear rank approximation of $\mathcal{A}$ with different multilinear ranks $\{P,P,P\}$.
\begin{figure}[htb]
	\centering
	\includegraphics[width=7in, height=2.5in]{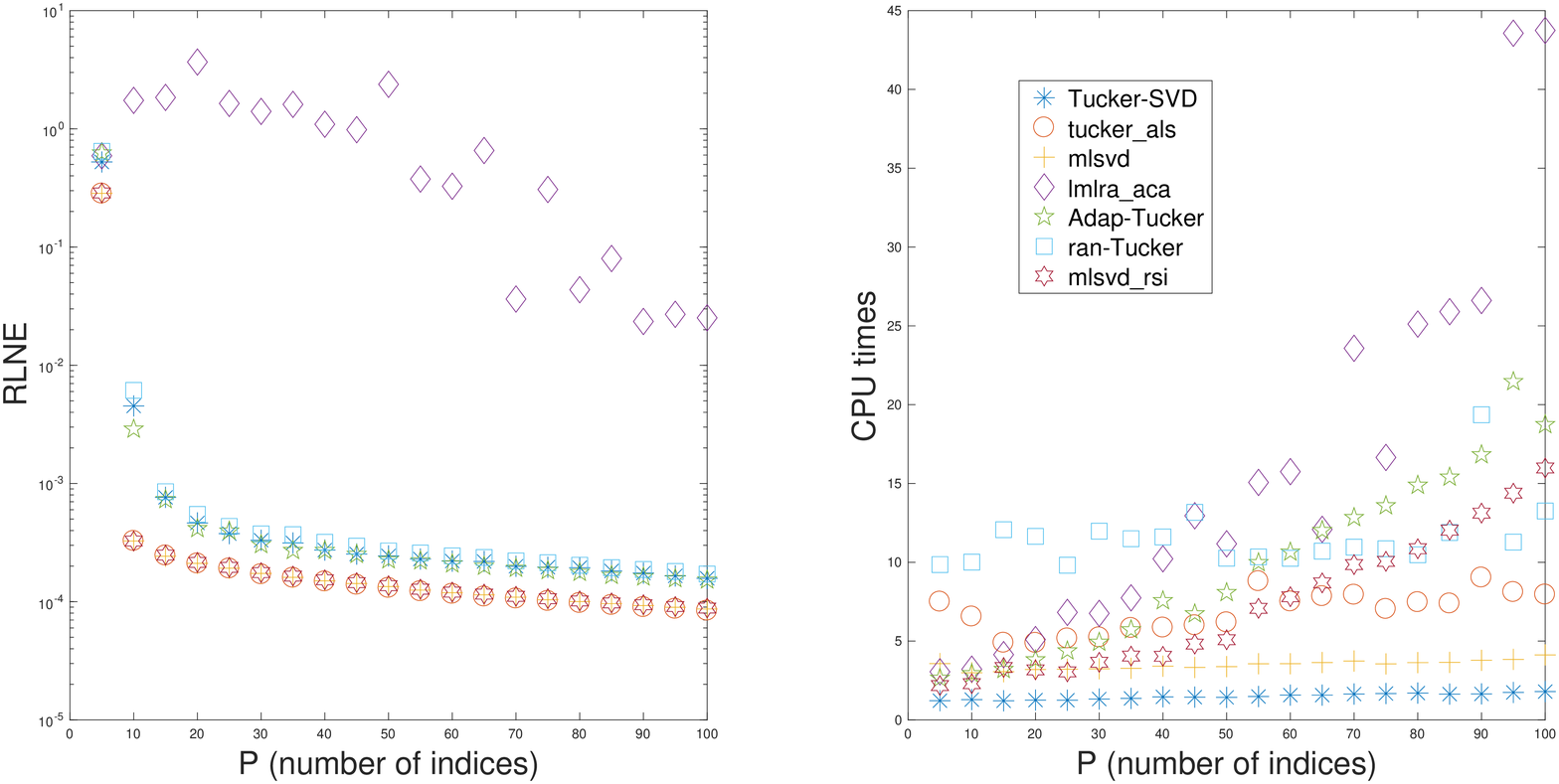}\\
	\caption{Numerical simulation results of applying Tucker-SVD, tucker\_als, mlsvd, lmlra\_aca, Adap-Tucker, ran-Tucker and mlsvd\_rsi to the sparse tensor $\mathcal{A}$ with $P=5,10,\dots,100$ and $I=400$.}\label{SUB:fig3}
\end{figure}

\subsection{Tucker form tensors plus the white noise}
Let $\mathcal{A}\in\mathbb{R}^{I\times I\times I}$ be given in the Tucker form \cite{tensor_cur_2010}
$
	\mathcal{A}=\mathcal{G}\times_1\mathbf{B}_1\times_2\mathbf{B}_2\times_3\mathbf{B}_3
$,
where the entries of $\mathcal{G}\in\mathbb{R}^{100\times 100\times 100}$  and $\mathbf{B}_n\in\mathbb{R}^{I\times 100}\ (n=1,2,3)$ are i.i.d. Gaussian variables with zero mean and unit variance. The form of this test tensor $\mathcal{C}$ is given as $\mathcal{C}=\mathcal{A}+\beta\mathcal{N}$, where $\mathcal{N}\in\mathbb{R}^{I\times I\times I}$ is an unstructured perturbation tensor with different noise level $\beta$. The following signal-to-noise ratio (SNR) measure will be used
\begin{equation*}
	{\rm SNR}\ [{\rm dB}]=10\log\left(\frac{\|\mathcal{B}\|_F^2}{\|\beta\mathcal{N}\|_F^2}\right).
\end{equation*}

The FIT value for approximating the tensor $\mathcal{C}$ is defined by
${\rm FIT}=1-{\rm RLNE}$,
where ${\rm RLNE}$ is given in (\ref{SUB:eqn21}). We assume that $I=400$. We compute a low multilinear rank approximation of $\mathcal{C}$ with the given multilinear rank $\{100,100,100\}$ using Tucker-SVD, tucker\_als, mlsvd, lmlra\_aca, ran-Tucker and mlsvd\_rsi. Figure \ref{SUB:fig4} compares efficiency and accuracy of different methods on $\mathcal{C}$ with different SNR values.

\begin{figure}[htb]
	\centering
	\includegraphics[width=7in, height=2.5in]{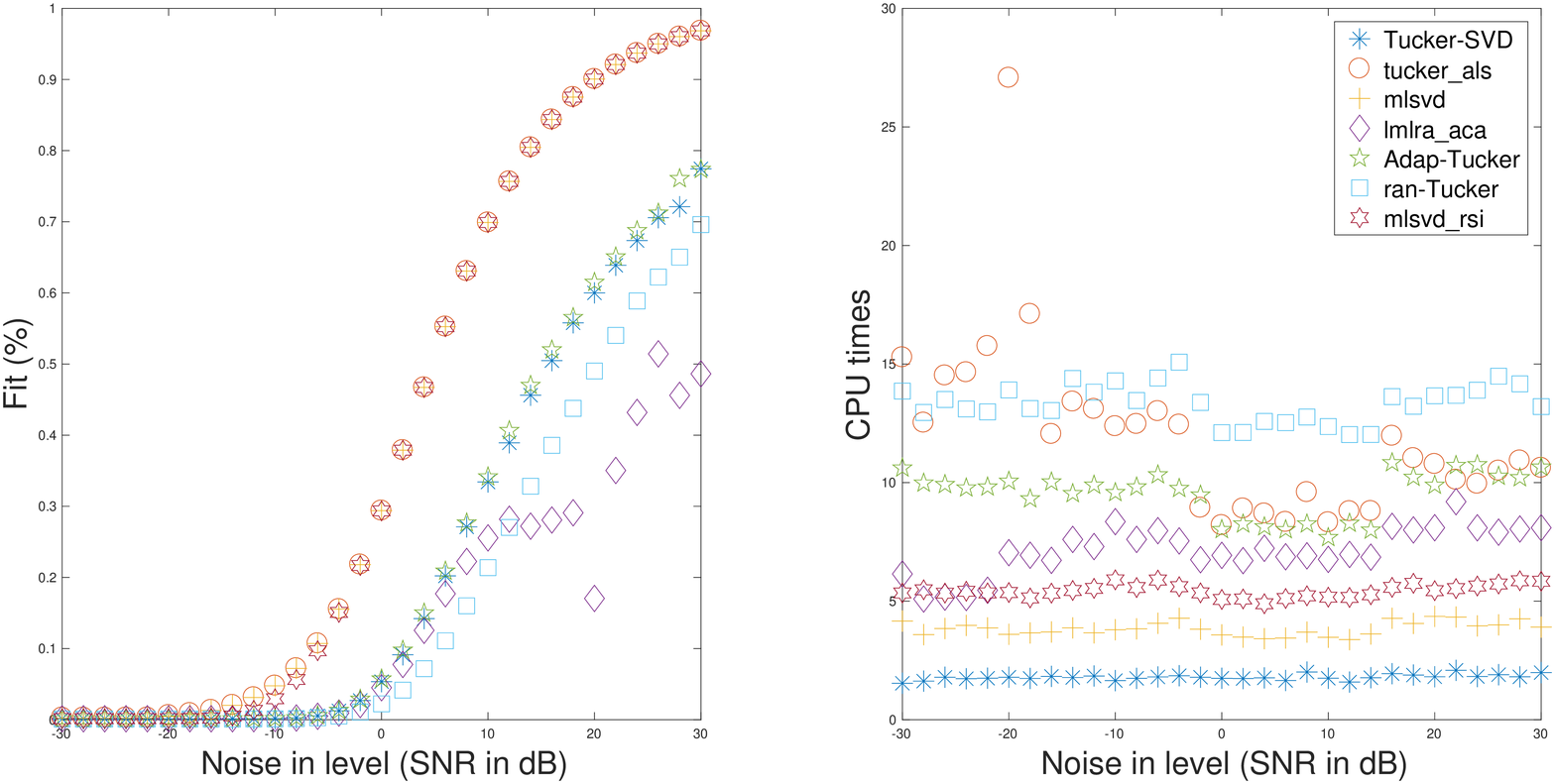}\\
	\caption{Numerical simulation results of applying Tucker-SVD, tucker\_als, mlsvd, lmlra\_aca, Adap-Tucker and mlsvd\_rsi to the sparse tensor $\mathcal{C}$ with different SNRs and $I=400$.}\label{SUB:fig4}
\end{figure}

\begin{remark}
	As shown in Figure {\rm\ref{SUB:fig4}}, for each algorithm, the CPU time of different SNRs is not very different. The reason is that the size of $\mathcal{C}$ is $400\times 400\times 400$ and $P=100$.
\end{remark}

As shown in Figures \ref{SUB:fig1}, \ref{SUB:fig2}, \ref{SUB:fig3} and \ref{SUB:fig4} and in terms of CPU time, Tucker-SVD is the fastest one; in terms of RLNE and FIT, Tucker-SVD is comparable to tucker\_als, mlsvd and mlsvd\_rsi.

\subsection{Handwritten digit classification}

In handwritten digits classification, we train a classifier to classify new unlabeled
images. Savas and Eld\'{e}n \cite{savas2007handwritten} presented two algorithms for handwritten digit classification based on HOSVD. To reduce the training time, Vannieuwenhoven {\it et al.} \cite{vvm_2012_sisc} presented a more efficient ST-HOSVD algorithm. In this section, we compare the performance of Tucker-SVD tucker\_als, mlsvd, ran-Tucker and mlsvd\_rsi on the MNIST database\footnote{The database can be obtained from http://yann.lecun.com/exdb/mnist/.} \cite{lecun1998gradient}, which contains 60,000 training images and 10,000 test images. Here the digit size is $28\times 28$ pixels with the same intensity range. The digit distribution is
given in Table \ref{SUB:tab1}. As seen in Table \ref{SUB:tab1}, The training images are unequally distributed over the ten
classes. Therefore, we restricted the number of training images in every class is less than or equal
to 5421.
\begin{table}
	\centering
	\begin{tabular}{cccccccccccc}
		\hline
		& 0 & 1 & 2 & 3 & 4 & 5 & 6 & 7 & 8 & 9 & Total \\
		\hline
		Train & 5923 & 6742 & 5958 & 6131 & 5842 & 5421 & 5918 & 6265  & 5851 &  5949 & 60000 \\
		\hline
		Test & 940 & 1135 & 1032 & 1010 &  982 &  892 &  958 &  1028 &  974  & 1009  & 10000 \\
		\hline
	\end{tabular}
	\caption{The digit distribution in the MNIST data set.}\label{SUB:tab1}
\end{table}

The training set can be represented by a tensor $\mathcal{A}\in\mathbb{R}^{786\times K\times 10}$, where $K\leq 5421$, this assumption is the same as in \cite{savas2007handwritten}. The first mode is the texel mode. The second mode corresponds to the training images. The third mode corresponds to different classes. Here we use Algorithm 2 in \cite{savas2007handwritten} to handwritten digit classification. We use various algorithms to obtain an approximation $\mathcal{A}\approx\mathcal{G}\times_1 \mathbf{U}\times_2 \mathbf{V}\times_3 \mathbf{W}$ with $\mathcal{G}\in\mathbb{R}^{65\times 142\times 10}$.

For $K=2500$, the related results are summarized in Table \ref{SUB:fig2}. In terms of CPU time, Tucker-SVD is the fastest one. In term of classification accuracy, Tucker-SVD is comparable to Tucker-ALS, mlsvd, Adap-Tucker, ran-Tucker and mlsvd\_rsi.

\begin{remark}
	By using the algorithms in {\rm\cite{savas2007handwritten}} to handwritten digit classification, the factor matrices are orthonormal. Hence we do not use Tucker-RRLU for handwritten digit classification.
\end{remark}
\begin{table}[htb]
	\centering
	\begin{tabular}{cccc}
		\hline
		& TT [sec] & RLNE & CA [\%]  \\
		\hline
		Tucker-SVD & 0.8200 & 0.4468 & 91.49  \\
		\hline
		tucker\_als & 20.0400 & 0.3128 & 93.11  \\
		\hline
		mlsvd & 13.0600 & 0.3140 & 93.18  \\
		\hline
		Adap-Tucker & 1.8900 & 0.4628 & 92.50  \\
		\hline
		ran-Tucker & 44.0500 & 0.4418 & 92.02  \\
		\hline
		mlsvd\_rsi & 3.9700 & 0.4418 & 93.50  \\
		\hline
	\end{tabular}
	\caption{Comparison on handwritten digits classification. Note that `TT' and `AC' denote the training time and classification accuracy, respectively, and floating point numbers in each example have four
		decimal digits.}\label{SUB:tab2}
\end{table}

For different $K$, the results are shown in Figure \ref{SUB:fig5}. From this figure, in terms of running time, Tucker-ALS is the most expensive one; in term of classification accuracy, Tucker-SVD, Tucker-ALS, mlsvd, Adap-Tucker and mlsvd\_rsi are comparable.

\begin{figure}[htb]
	\centering
	\includegraphics[width=7in, height=2.5in]{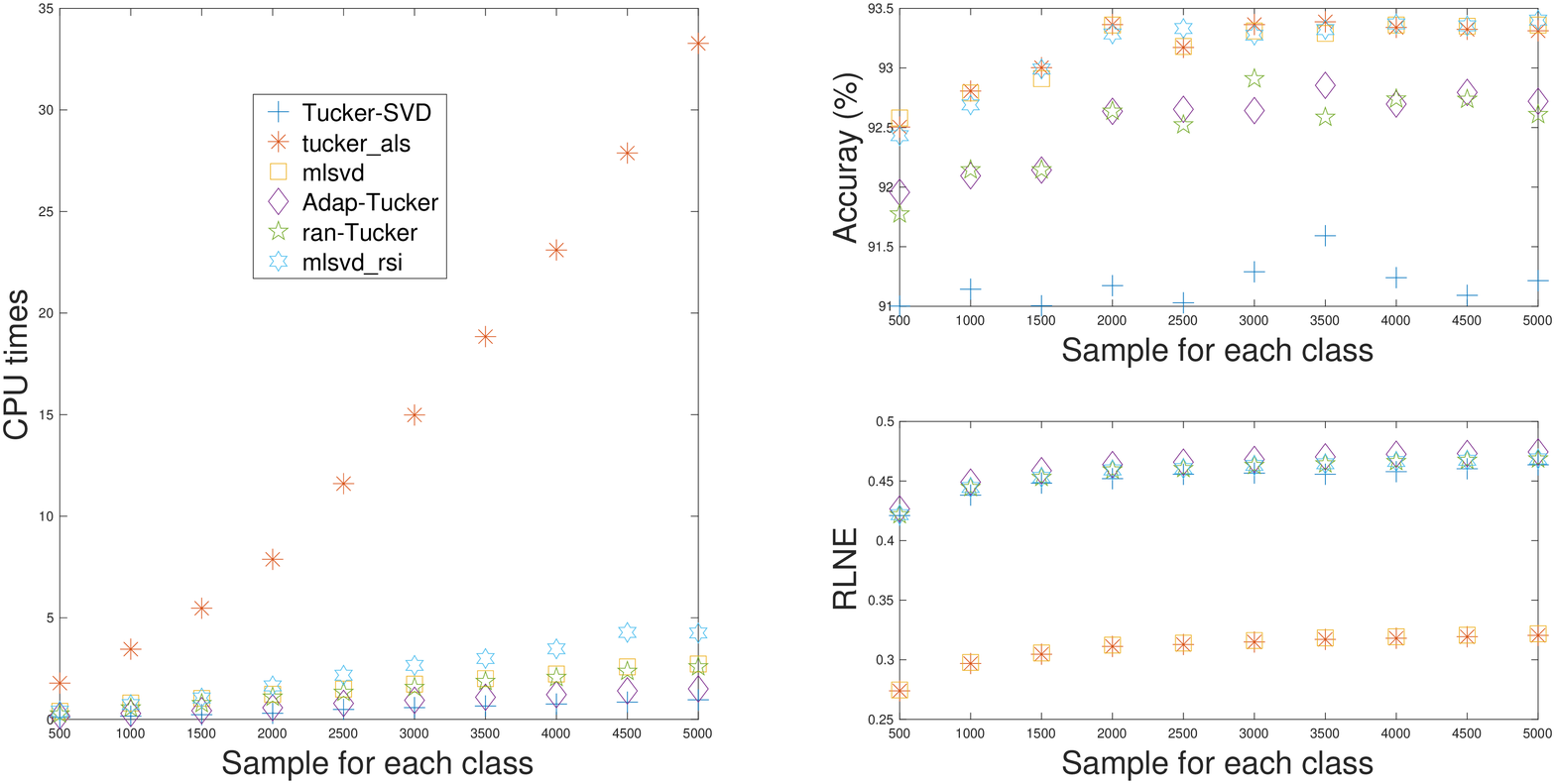}\\
	\caption{Comparison on handwritten digits classification with $K=500,1000,\dots,5000$.}\label{SUB:fig5}
\end{figure}
\subsection{Generalization for the case of $N=4$}

For the given multilinear rank $\{\mu_1,\mu_2,\mu_3,\mu_4\}$ of $\mathcal{A}\in\mathbb{R}^{I_1\times I_2\times I_3\times I_4}$, the generalization of Algorithm \ref{SUB:alg8} is summarized in the following algorithm. Without loss of generality, Algorithm \ref{SUB:alg9} is also denoted as Tucker-SVD.
\begin{algorithm}[htb]
	\caption{The proposed randomized algorithm for low multilinear rank approximations with $N=4$}
	\label{SUB:alg9}
	\begin{algorithmic}[1]
		\STATEx {\bf Input}: A tensor $\mathcal{A}\in \mathbb{R}^{I_1\times I_2\times I_3\times I_4}$ to decompose, the desired multilinear rank $\{\mu_1,\mu_2,\mu_3,\mu_4\}$, $L_{4,1}L_{4,2}L_{4,3}\geq\mu_4+K$, $L_{3,1}L_{3,2}L_{3,4}\geq\mu_3+K$, $L_{2,1}L_{2,3}L_{2,4}\geq\mu_2+K$, $L_{1,2}L_{1,3}L_{1,4}\geq\mu_1+K$ number of columns to use, and a processing order $\mathbf{p}\in\mathbb{S}_4$, where $K$ is a oversampling parameter.
		\STATEx {\bf Output}: Four orthonormal matrices $\mathbf{Q}_n$ such that $\|\mathcal{A}\times_1 (\mathbf{Q}_1\mathbf{Q}_1^\top)\times_2 (\mathbf{Q}_2\mathbf{Q}_2^\top)
		\times_3 (\mathbf{Q}_3\mathbf{Q}_3^\top)\times_4 (\mathbf{Q}_4\mathbf{Q}_4^\top)-\mathcal{A}\|_F\leq
		\sum_{n=1}^4O(\Delta_{\mu_n+1}(\mathbf{A}_{(n)}))$.
		\STATE Set  the temporary tensor: $\mathcal{C}=\mathcal{A}$.
		\FOR {$n=p_1,p_2,p_3,p_4$}
		\STATE Form three real matrices $\mathbf{G}_{n,m}\in\mathbb{R}^{L_{n,m}\times I_m}$ whose entries are i.i.d. Gaussian random variables of zero mean and unit variance, where $m=1,2,3,4$ and $m\neq n$.
		\STATE Compute the product tensor
		\begin{equation*}
			\mathcal{B}_n=\mathcal{C}\times_1\mathbf{G}_{n,1}\dots\times_{m-1}\mathbf{G}_{n,m-1}
			\times_{m+1}\mathbf{G}_{n,m+1}\dots\times_4\mathbf{G}_{n,4}.
		\end{equation*}
		\STATE Form the mode-$n$ unfolding $\mathbf{B}_{n,(n)}$ of the tensor $\mathcal{B}_n$.
		\STATE For the $\mathbf{B}_{n,(n)}$, find a real $I_n\times \mu_n$ matrix $\mathbf{Q}_n$ whose columns are orthonormal, such that there exists a real $\mu_n\times \prod_{m=1,m\neq n}^4L_{n,m}$ matrix $\mathbf{S}_n$ for which
		$$\|\mathbf{Q}_n\mathbf{S}_n-\mathbf{B}_{n,(n)}\|_2\leq\sigma_{\mu_n+1}(\mathbf{B}_{n,(n)}).$$
		\STATE Set $I_n=\mu_n$ and $\mathbf{Q}_n=\mathbf{Q}_n(:,1:\mu_n)$.
		\STATE Compute $\mathcal{C}=\mathcal{C}\times_n\mathbf{Q}_n^\top$.
		\ENDFOR
	\end{algorithmic}
\end{algorithm}

\begin{figure}
	\centering
	\includegraphics[width=7in, height=2.5in]{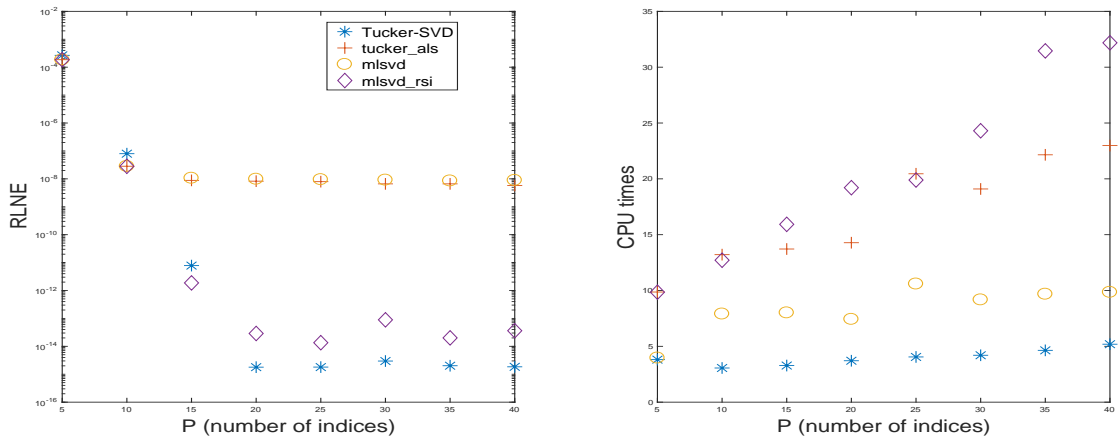}\\
	\caption{Numerical simulation results of applying Tucker-SVD, tucker\_als, mlsvd and mlsvd\_rsi to $\mathcal{A}$ with $P=5,10,\dots,40$ and $I=100$.}\label{SUB:fig1add}
\end{figure}

For a given low multilinear rank approximation $\widehat{\mathcal{A}}=\mathcal{A}\times_{1}({\bf S}_{1}{\bf S}_{1}^\top)\times_{2}({\bf S}_{2}{\bf S}_{2}^\top)\times_{3}({\bf S}_{3}{\bf S}_{3})^\top\times_{4}({\bf S}_{4}{\bf S}_{4}^\top)$ of $\mathcal{A}\in\mathbb{R}^{I_1\times I_2\times I_3\times I_4}$, where the matrices ${\bf S}_n\in\mathbb{R}^{I\times \mu}$ are derived form the desired numerical algorithms, its relative least normalized error (RLNE) is defined as
\begin{equation*}
	{\rm RLNE}=\|\mathcal{A}-\widehat{\mathcal{A}}\|_F/\|\mathcal{A}\|_F.
\end{equation*}

Now we consider the first test tensor generated by sampling a smooth function as follows
\begin{equation*}
	a_{ijkl}=\frac{1}{i+j+k+l},
\end{equation*}
with $i,j,k,l=1,2,\dots,I$.

Suppose that $I=100$. We compute a low multilinear rank approximation of $\mathcal{A}$ with multilinear rank $\{P,P,P,P\}$ using Tucker-SVD, tucker\_als, mlsvd and mlsvd\_rsi, respectively.

Figure \ref{SUB:fig1add} compares efficiency and accuracy of different methods on $\mathcal{A}$. In terms of CPU time, Tucker-SVD is the fastest; in terms of RLNE, Tucker-SVD is comparable to mlsvd\_rsi.

Another test tensor $\mathcal{B}\in\mathbb{R}^{I\times I\times I}$ is a sparse tensor, which is defined as \cite{s_2016_simax,sorensen2016a}
\begin{equation*}
	\mathcal{B}=\sum_{j=1}^{10}\frac{1000}{j}(\mathbf{x}_j
	\circ\mathbf{y}_j\circ\mathbf{z}_j\circ\mathbf{w}_j)+
	\sum_{j=11}^I\frac{1}{j}(\mathbf{x}_j
	\circ\mathbf{y}_j\circ\mathbf{z}_j\circ\mathbf{w}_j)
\end{equation*}
where $\mathbf{x}_j,\mathbf{y}_j,\mathbf{z}_j,\mathbf{w}_j\in\mathbb{R}^I$ are sparse vectors with nonnegative entries. In MATLAB,
\begin{equation*}
\begin{cases}
\mathbf{x}_j={\rm sprand(I,1,0.015)},\quad\mathbf{y}_j={\rm sprand(I,1,0.025)},\\
\mathbf{z}_j={\rm sprand(I,1,0.035)},\quad\mathbf{w}_j={\rm sprand(I,1,0.045)}.
\end{cases}
\end{equation*}
Here we assume that $I=100$.

  Figure \ref{SUB:fig2add} shows three results of RLNE and CPU time for Tucker-SVD, tucker\_als, mlsvd, and mlsvd\_rsi used to find a low multilinear rank approximation of $\mathcal{B}$ with multilinear rank $\{P,P,P,P\}$.  In terms of CPU time, Tucker-SVD is the fastest and in terms of RLNE, Tucker-SVD is  comparable to tucker\_als, mlsvd, and mlsvd\_rsi.
\begin{figure*}[htb]
	\centering
	\includegraphics[width=7in, height=2.5in]{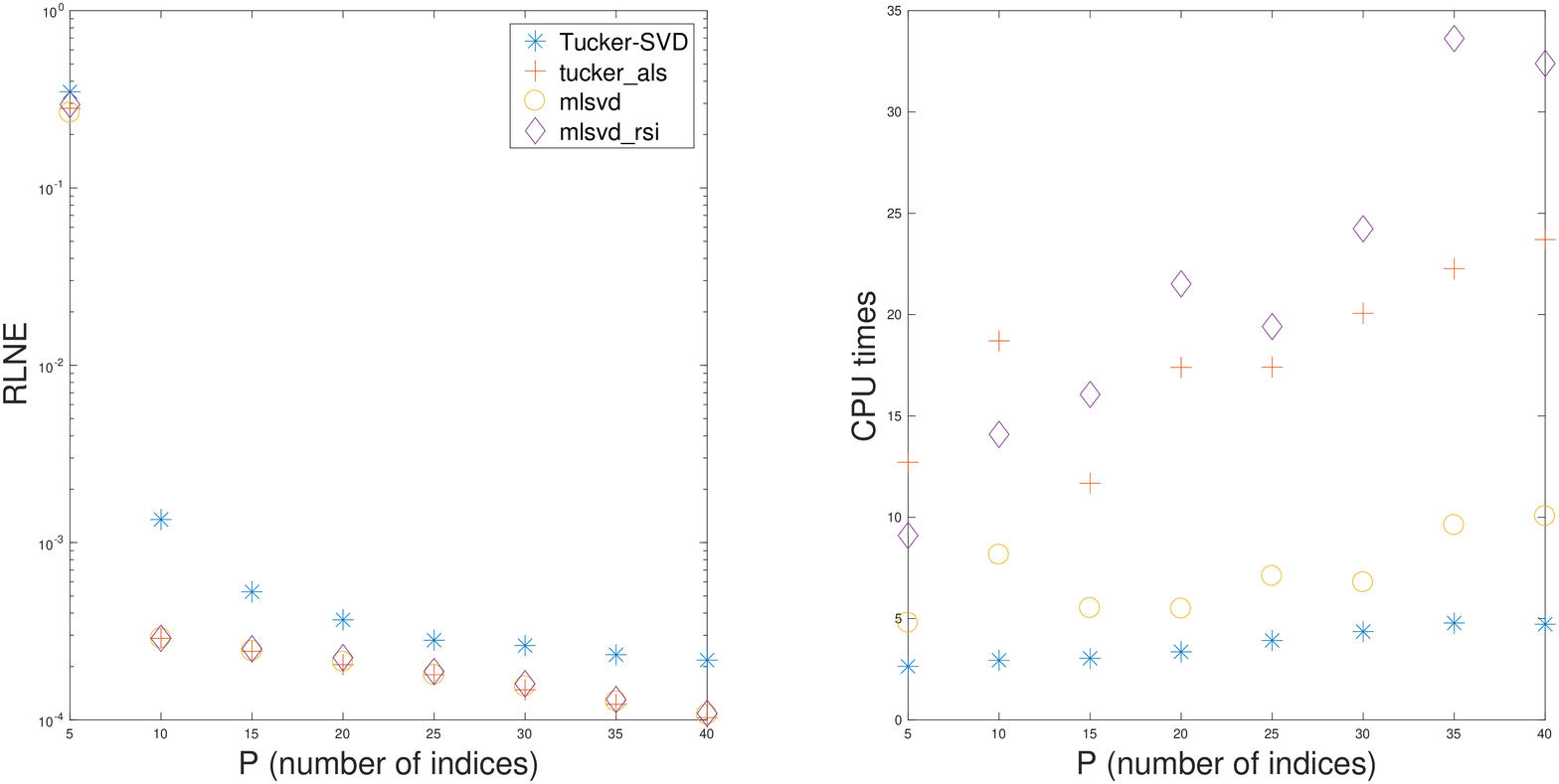}\\
	\caption{Numerical simulation results of applying Tucker-SVD, tucker\_als, mlsvd, and mlsvd\_rsi to the sparse tensor $\mathcal{B}$ with $P=5,10,\dots,40$ and $I=100$.}\label{SUB:fig2add}
\end{figure*}

\section{Conclusion and discussion}

\label{SUB:sect7}
In this paper, based on the SVD and random projections, we propose a randomized algorithm Tucker-SVD for low multilinear rank approximations of tensors. Numerical examples illustrate that Tucker-SVD is fastest in terms of CPU time and the low multilinear rank approximation derived by Tucker-SVD can be used as a criterion for judging the merits and demerits of other algorithms. The error bound in Theorem \ref{SUB:thm6} is a rough estimation. Improving this bound would be an interesting topic. Numerical examples illustrate that in terms of RLNE, Tucker-SVD is worse than these algorithms in some cases. In order to reduce RLNE obtained by Tucker-SVD, Che {\it et al.} \cite{che2018randomized1} obtain another randomized algorithm for solving Problem \ref{SUB:prob1} by combining Tucker-SVD and power scheme.

Che and Wei \cite{che2018randomized} consider the adaptive randomized algorithm for the approximate tensor train decomposition. One of the future considerations is to design more effective randomized algorithms for the approximate tensor train decomposition, based on the idea of the proposed algorithms in this paper. The tensor train structure is a special case of the Hierarchical Tucker decomposition. Our second  consideration is to design randomized algorithms for the Hierarchical Tucker approximation of tensors.

{\small
\bibliographystyle{siam}
\bibliography{random_algorithm_for_tucker_decomposition}
}
\end{document}